\documentclass[11pt]{amsart}

\usepackage{amsmath, amsthm, amssymb}

\usepackage{lmodern}
\usepackage{palatino}                       % text font
\usepackage[bigdelims,vvarbb]{newpxmath}    % math font 
\usepackage[scaled=0.95]{inconsolata}       % Sans font 
\usepackage[T1]{fontenc}                    % font encoding 

\usepackage{comment}

\usepackage[abs]{overpic}		  % overpic automatically loads graphicx
\usepackage{import}
\usepackage{transparent}

\usepackage{xcolor}
\definecolor{lred}{RGB}{226, 106, 106}
\definecolor{nred}{RGB}{237, 28, 36}
\definecolor{ddred}{RGB}{255, 0, 0}
\definecolor{lblue}{RGB}{52, 152, 219}
\definecolor{nblue}{RGB}{0, 174, 239}
\definecolor{lyellow}{RGB}{232, 197, 91}
\definecolor{dgreen}{RGB}{0, 148, 68}
\definecolor{l1yellow}{RGB}{217, 224, 33}
\definecolor{l2yellow}{RGB}{216, 177, 64}
\definecolor{lgrey}{RGB}{179, 179, 179}
\definecolor{indigo}{rgb}{0.29, 0.0, 0.51}  % custom colors
\usepackage[colorlinks, urlcolor=indigo, linkcolor=indigo, citecolor=indigo]{hyperref}

\usepackage[hcentering, vcentering, total={5.9in, 8.2in}]{geometry}  % 1.4 vertical margin; 1.3 horizontal margin

\usepackage{tikz}
\usetikzlibrary{trees}

\theoremstyle{plain}
\newtheorem{theorem}{Theorem}
\newtheorem{corollary}[theorem]{Corollary}
\newtheorem{proposition}[theorem]{Proposition}
\newtheorem{lemma}[theorem]{Lemma}

\theoremstyle{definition}
\newtheorem{definition}[theorem]{Definition}

\theoremstyle{remark}
\newtheorem{remark}[theorem]{Remark}
\newtheorem{example}[theorem]{Example}

\numberwithin{theorem}{section}

\newcommand{\dfn}[1]{{\em #1}}        % definition
\newcommand{\R}{\mathbb{R}}           % the real numbers
\newcommand{\Z}{\mathbb{Z}}           % the integers
\newcommand{\C}{\mathbb{C}}           % the complex numbers
\DeclareMathOperator{\bd}{\partial}   % boundary
\newcommand{\modp}[1]{\;(\!\!\!\!\!\!\mod #1)}      % mod for display math. \pmod is for inline math 
\makeatletter
\newcommand*\bigcdot{\mathpalette\bigcdot@{0.6}}
\newcommand*\bigcdot@[2]{\mathbin{\vcenter{\hbox{\scalebox{#2}{$\m@th#1\bullet$}}}}}
\makeatother

\DeclareMathOperator\tb{tb}                               % Thurston-Bennequin
\DeclareMathOperator\rot{rot}                             % rotation
\DeclareMathOperator{\cp}{\mathbb{C}{{P}}^2}  % CP2
\DeclareMathOperator{\Cont}{Cont}         % group of contactomorphisms 

\DeclareFontFamily{U} {cmr}{}
\DeclareFontShape{U}{cmr}{m}{n}{
  <-6> cmr5
  <6-7> cmr6
  <7-8> cmr7
  <8-9> cmr8
  <9-10> cmr9
  <10-12> cmr8
  <12-> cmr9}{}
\DeclareSymbolFont{Xcmr} {U} {cmr}{m}{n}
\DeclareMathSymbol{\Phi}{\mathord}{Xcmr}{8}

\begin{document}

% title
\title[Small symplectic caps and embeddings of homology balls]{Small symplectic caps and embeddings of \\ homology balls in the complex projective plane}

\author{John Etnyre}

\author{Hyunki Min}

\author{Lisa Piccirillo}

\author{Agniva Roy}

\address{School of Mathematics \\ Georgia Institute of Technology \\ Atlanta, GA}
\email{etnyre@math.gatech.edu}

\address{Department of Mathematics \\ University of California \\ Los Angeles, CA}
\email{hkmin27@math.ucla.edu}

\address{Department of Mathematics \\ Massachusetts Institute of Technology \\ Cambridge, MA}
\email{piccirli@mit.edu}

\address{Department of Mathematics \\ Boston College \\ Chestnut Hill, MA}
\email{agniva.roy@bc.edu}

%\subjclass[2020]{57K43}

% abstract
\begin{abstract}
 We present a handlebody construction of small symplectic caps, and hence of small closed symplectic $4$-manifolds. We use this to construct handlebody descriptions of symplectic embeddings of rational homology balls in $\cp$, and thereby provide the first examples of (infinitely many) symplectic handlebody decompositions of a closed symplectic $4$-manifold. Our constructions provide a new topological interpretation of almost toric fibrations of $\cp$ in terms of symplectic handlebody decompositions.
\end{abstract}

\maketitle
%\tableofcontents
\vspace{-15pt}

%%%%%%%%%%%%%%%%%%%%%%%%%%%%%%%%%%%
\section{Introduction}
%%%%%%%%%%%%%%%%%%%%%%%%%%%%%%%%%%%

The literature contains many constructions of closed symplectic $4$-manifolds, for example as complex submanifolds of $\C P^n$,  via symplectic reduction,  as toric fibrations, or as Lefschetz pencils or fibrations.  To be readily compatible however with the tools commonly used by smooth 4-manifold topologists, it is desirable to have a working theory of how to build closed symplectic manifolds out of handles.  

We already have a good understanding of handlebody constructions of symplectic fillings by work of Eliashberg \cite{Eliashberg:Stein}, Gompf \cite{Gompf:Stein} and Weinstein \cite{Weinstein:handle}. To get a handle description of a closed symplectic $4$-manifold,  one might want to glue such a filling to symplectic cap along a fixed contact $3$-manifold.  But there is presently no fully handle-theoretic construction for symplectic caps; because Weinstein $4$-manifolds only have handles of index at most  $2$, there are no Weinstein $3$- and $4$-handles. Moreover, existing constructions of symplectic caps, e.g. \cite{EtnyreHonda02a}, nearly always produce caps with large homology.  Developing a practical handle theoretic construction of small symplectic caps is the primary goal of this paper. 

Our construction of small symplectic caps relies on a technique of Gay~\cite{Gay:handle} from 2002. Gay's technique suggests building symplectic caps by attaching particular $2$-handles (called a \dfn{convex-concave handle}, see Section~\ref{subsec:handles}) to the convex boundary of a symplectic filling. As the name suggests, after attaching a convex-concave 2-handle one has a symplectic 4-manifold with concave boundary. That concave boundary can then be capped with an (upside down) Weinstein handlebody to obtain a closed symplectic manifold. 

Building closed symplectic manifolds in this way was the original intended purpose of Gay's technique, but to date it has not been carried out because it is difficult to identify the concave boundary produced after the convex-concave 2-handle attachment.  Even when one can identify the resulting contact 3-manifold, it is frequently overtwisted, and hence does not even admit a weak symplectic filling. In this paper, we use recent developments in contact 3-manifold topology, notably \cite{EMM:nonloose}, to get past these technical issues in certain circumstances.  

As a demonstration of how one might work with our symplectic caps in practice,  we use them to construct hypersurfaces of contact type in $\cp$ (equivalently, construct symplectic embeddings in $\cp$ that the hypersurfaces bound).  Understanding the settings in which 3-manifolds embed in 4-manifolds is a hard problem with rich history in both the smooth and symplectic categories.  Perhaps the easiest setting to study is for the simplest 3-manifolds, lens spaces,  and the simplest closed 4-manifold in which they can embed, $\cp$.  In the smooth category little is known.  For example, it is unknown whether four distinct lens spaces can be (disjointly) embedded in $\cp$.  In contrast, in the symplectic category this problem is completely understood. In fact, Vianna \cite{Vianna:exotic} used almost toric fibrations and constructed a family of three disjoint lens spaces that are embedded in $\cp$ as hypersurfaces of contact type (equivalently, three rational homology balls that symplectically embed in $\cp$). Evans and Smith \cite{EvansSmith:markov} then proved that Vianna's are the only hypersurfaces of contact type among all families of lens spaces in $\cp$. Recently, Lisca and Parma \cite{LiscaParma:horiz2} gave a smooth interpretation of Vianna's embeddings using ``horizontal" handle decompositions (Section~\ref{subsec:horiz}). We give a new symplectic interpretation of Vianna's embeddings using symplectic handle decompositions. 

\begin{theorem} \label{thm:main1}
  For any Markov triple $(p_1, p_2, p_3)$ there exists a triple of integers $(q_1,q_2,q_3)$ such that the the three rational homology 4-balls $B_{p_i,q_i}$ are disjointly symplectically embedded in $\cp$.
  Furthermore, the Weinstein handle decompositions of the rational homology balls are sub-decompositions of an explicit symplectic handle decomposition of $\cp$.
\end{theorem}

See Section~\ref{subsec:Markov} for the definition of Markov triples, Figure~\ref{fig:Bpqintro} for the rational homology ball $B_{p,q}$ and Definition~\ref{def:decomposition} for the formal definition of a symplectic handle decomposition.

\begin{figure}[htbp]{\scriptsize
  \vspace{0.2cm}
  \begin{overpic}[width=.4\textwidth, tics=20]
  {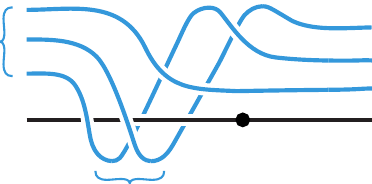}
   \put(54, -7){\color{lblue}$-p$}
   \put(-9, 64){\color{lblue}$q$}
   \put(150, 78){\color{lblue}$-1$}
  \end{overpic}}
  \vspace{0.2cm}
  \caption{A handlebody picture for the rational ball $B_{p,q}$ has boundary the lens space $L(p^2,pq-1)$.  The $-1$ framing on the $2$-handle is relative to the torus framing. Here and throughout, figures should be braid closed.}
  \label{fig:Bpqintro}
\end{figure}

Our techniques can be used to construct many other embeddings of contact lens spaces as hypersurfaces of contact type in $\cp\# n\overline{\cp}$, which we do not record here.  Ultimately,  we hope that with further study our techniques can be used to construct other, possibly exotic, small symplectic manifolds. 

We outline the proof of Theorem \ref{thm:main1} now. The main technical work lies in defining a convex-concave $2$-handle between standard contact structures on (connected sums of) lens spaces. To do this, we attach a convex-concave $2$-handle to the symplectization of a standard contact lens space to obtain a symplectic cobordism with two concave boundary components. One boundary component is the original lens space,  the other is a connected sum of two lens spaces. In Section~\ref{subsec:nonloose}, we show that the contact structure induced on the reducible boundary component is equivalent to the result of Legendrian surgery on some Legendrian torus knot in an overtwisted lens space; this allows us to conclude that the contact manifold is a connected sum of standard contact lens spaces. This cobordism can then be capped off by Weinstein fillings $B_{p_i,q_i}$ of standard contact lens spaces to obtain a closed symplectic manifold. We denote the resulting symplectic manifold by $(X_{p_1,p_2,p_3}, \omega_{p_1,p_2,p_3})$. In Section~\ref{sec:top} we use horizontal handlebody decompositions, introduced by Lisca and Parma \cite{LiscaParma:horiz}, to show that $X_{p_1,p_2,p_3}$ is diffeomorphic to $\cp$. A theorem of Taubes \cite{Taubes:SWGr} then guarantees that $(X_{p_1,p_2,p_3},\omega_{p_1,p_2,p_3})$ is deformation equivalent to $\cp$. 
 
In Section~\ref{sec:toric} we give two additional proofs that $X_{p_1,p_2,p_3}$ is diffeomorphic to $\cp$. The first of these, inspired by Vianna \cite{Vianna:exotic}, inductively identifies our spaces $X_{p_1,p_2,p_3}$  with $\cp$. The second uses a handlebody description of almost toric fibrations to exhibit the diffeomorphism. We conclude the paper by showing that this almost toric fibration structure on $X_{p_1,p_2,p_3}$ indeed agrees with the almost toric structures that Vianna used to build the original embeddings of the rational homology balls.  

\begin{theorem} \label{thm:main2}
  An almost toric fibration of $\cp$ for given Markov triple $(p_1, p_2, p_3)$ is compatible with the symplectic handlebody decomposition from Theorem~\ref{thm:main1}, {\it i.e.} the restriction on each of sub-handlebodies, $B_{p_i,q_i}$ and the pants cobordism is also an almost toric fibration. 
\end{theorem}

\noindent
{\bf Organization.}
In Section~\ref{sec:prelim}, we collect the background material we will need.  In Section~\ref{sec:geometry}, we build symplectic pants cobordisms between a standard contact lens space and a disjoint union of two standard contact lens spaces and construct the symplectic manifolds $X_{p_1,p_2,p_3}$ into which the rational homology balls $B_{p_i, q_i}$ embed. In Section~\ref{sec:top}, we draw a handle diagram for $X_{p_1.p_2,p_3}$ and prove that the result is $\cp$ equipped with the standard symplectic structure. This proves Theorem~\ref{thm:main1}. In Section~\ref{sec:toric}, we exhibit two other proofs that $X_{p_1,p_2,p_3}$ is $\cp$ and prove Theorem~\ref{thm:main2}. 

\noindent
{\bf Conventions.} 
The lens space $L(p,q)$ is defined to be the $-p/q$ Dehn surgery on the unknot. We define $B_{p,q}$ to be a smoothing of the cyclic quotient singularity of type $(p^2,pq-1)$.  For a handle diagram description of this manifold, see Figure \ref{fig:Bpq}. 

\section*{Acknowledgement}
JE and AR were partially supported by DMS-1906414 and DMS-2203312. JE was also partially supported by the Georgia Institute of Technology's Elaine M. Hubbard Distinguished Faculty Award.  LP was partially supported in part by a Sloan Fellowship, a Clay Fellowship, and the Simons collaboration ``New structures in low-dimensional topology''. Part of this work was done during Graduate Student Topology and Geometry Conference in 2022. JE, HM and AR are grateful to Georgia Tech for the support during this conference. The authors thank Paolo Lisca and Andrea Parma for helpful correspondence about Theorem~\ref{thm:LiscaParma}, and Nicki McGill, B\"ulent Tosun, and Morgan Weiler for useful conversations about several aspects of this paper. The authors also appreciate Charles Livingston for a helpful comment on a rational open book decomposition.  The authors are grateful to the referee for helpful comments.

%%%%%%%%%%%%%%%%%%%%%%%%%%%%%%%%%%%%%%%%%%%
\section{Preliminaries}\label{sec:prelim}
%%%%%%%%%%%%%%%%%%%%%%%%%%%%%%%%%%%%%%%%%%%

In this section we review the background results necessary for our main results. In Section~\ref{subsec:Markov}, we recall the definition of a Markov triple and discuss how to generate all such triples. Transverse surgery is reviewed in Section~\ref{subsec:surgery} and we discuss torus knots in lens spaces in Section~\ref{subsec:torus}. In Section~\ref{subsec:robd}, we recall the definition of and basic facts about rational open book decompositions. Various types of symplectic handle attachments are discussed in Section~\ref{subsec:handles} while Weinstein rational homology balls are built in Section~\ref{subsec:rhb}. Finally we discuss contact structures on lens spaces in Section~\ref{subsec:lens}.

%-----------------------------------------------------
\subsection{Markov triples and the Markov tree}\label{subsec:Markov}
%-----------------------------------------------------
A \dfn{Markov triple} is a triple of positive integers $(p_1, p_2, p_3)$ satisfying 
\[
  p_1^2 + p_2^2 + p_3^2 = 3p_1p_2p_3.
\] 
We note that if $(p_1, p_2, p_3)$ is a Markov triple, then so are the triples $(p_2,p_3,3p_2p_3-p_1)$ and $(p_1,p_3,3p_1p_3-p_2)$ and these are called \dfn{mutations} of the original triple. We can use these relation to build the \dfn{Markov tree}. This is a binary tree with Markov triples $(p_1,p_2,p_3)$ as vertices for $p_1 \leq p_2 \leq p_3$ and an edge connecting two vertices related by mutation. Specifically the root of the tree is $(1,1,1)$ which has a single child vertex $(1,1,2)$. The vertex $(1,1,2)$ also has a single child vertex $(1,2,5)$. Any other vertex $(p_1,p_2,p_3)$ has two child vertices; the left child is $(p_2,p_3,3p_2p_3-p_1)$, and the right child is $(p_1,p_3,3p_1p_3-p_2)$, see Figure~\ref{fig:markov-tree}. 

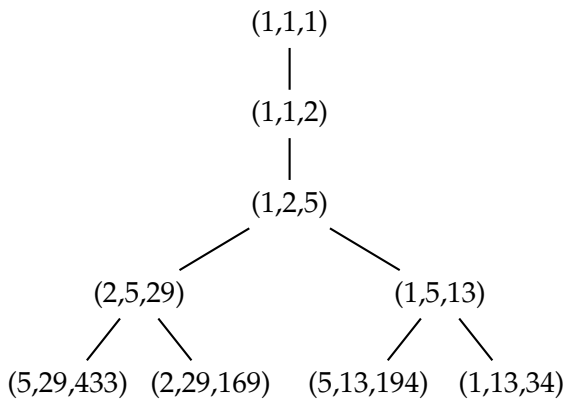
\begin{figure}[htbp]
\centering
\vspace{0.2cm}
\begin{tikzpicture}[thick, level distance=1.2cm,
  level 3/.style={sibling distance=4cm},
  level 4/.style={sibling distance=1.9cm}]
  \node {(1,1,1)}
    child {node {(1,1,2)}
      child {node {(1,2,5)}
        child {node {(2,5,29)}
          child {node {(5,29,433)}}
          child {node {(2,29,169)}}
        }
        child {node {(1,5,13)}
          child {node {(5,13,194)}}
          child {node {(1,13,34)}}
        }
      }
    }; 
\end{tikzpicture}
\caption{The Markov tree.}
\label{fig:markov-tree}
\end{figure}

%-----------------------------------------------------
\subsection{Transverse surgery}\label{subsec:surgery}
%-----------------------------------------------------
Let $K$ be a transverse knot in a contact $3$-manifold $(M,\xi)$. There exist polar coordinates $(r,\theta,\phi)$ on a neighborhood of $K$ such that the contact form $\alpha$ of $\xi$ can be written as
\[
  \alpha = r^2\,  d\theta + d\phi  
\]
where $K$ is identified with the $\phi$-axis and $r \in [0, R)$ for some $R > 0$. For any negative rational number $a$ with $\sqrt{-1/a} < R$, we call $S_a = \{r \leq \sqrt{-1/a}\}$ a \dfn{standard neighborhood of $K$ with slope $-1/a$}. Notice that the characteristic foliation on $\bd S_a$ is the linear foliation with slope $a$.

A {\it transverse surgery} on $(K, S_a)$ is a surgery operation to produce a new contact manifold. There are two types of transverse surgeries: \dfn{admissible} and \dfn{inadmissible} transverse surgeries. For more details, see \cite{BaldwinEtnyre:transverse,Conway:transverse}. In general, the resulting contact structure of transverse surgery depends on the choice of a neighborhood $S_a$. In particular, we can only perform admissible transverse $s$-surgery for $s < a$ (notice that the polar coordinates on $S_a$ determines the framing of $K$). If there is an obvious choice of a neighborhood for $K$, then we omit $S_a$ and just talk about admissible transverse surgery on $K$.  

Conway \cite{Conway:transverse} showed that inadmissible transverse surgery is equivalent to some contact surgery on its Legendrian approximations. In some cases, this is also true for admissible transverse surgery. This is explored in \cite[Lemma~3.16]{BaldwinEtnyre:transverse}, and a simple case of that theorem yields the following result. 

\begin{proposition} \label{prop:admissible-Legendrian}
  Let $S_a$ be a standard neighborhood of the transverse knot $K$. If $s = \lfloor a \rfloor - 1$, then admissible transverse $s$-surgery on $K$ is equivalent to Legendrian surgery on a Legendrian approximation of $K$ in $S_a$. 
\end{proposition}

%------------------------------------------------
\subsection{Convex surface basics}\label{subsec:convexBasic}
%------------------------------------------------
In this section, we will recall the basic facts about convex surfaces that will be needed below. A more detailed discussion, in terms similar to those used here, can be found in \cite{EtnyreHonda01} along with references to where the results first appeared. 

Recall that a \dfn{contact vector field} in a contact 3-manifold $(M,\xi)$ is a vector field whose flow preserves the contact structure $\xi$. An embedded surface $\Sigma$ in $(M,\xi)$ is \dfn{convex} if there exists a contact vector field $X$ transverse to $\Sigma$. If $\Sigma$ is a surface with boundary, then we assume $\bd\Sigma$ to be Legendrian. According to {Giroux}, any closed embedded surface can be $C^\infty$-perturbed to be convex. Kanda showed that if a surface with Legendrian boundary has the contact planes twisting non-positively along the boundary, then the surface may be perturbed (rel boundary) to be convex. 

Given a convex surface $\Sigma$, the \dfn{dividing set} $\Gamma_{\Sigma}$ is a set of points on $\Sigma$ where the contact vector field $X$ is tangent to $\xi$. The dividing set $\Gamma_{\Sigma}$ is an embedded multicurve on $\Sigma$ and its isotopy class is independent of the choice of a contact vector field. The dividing set divides $\Sigma$ into two regions: 
\[  
  \Sigma \setminus \Gamma_{\Sigma} = R_+ \cup R_-
\]
where $R_+ = \{p \in \Sigma : \alpha(X_p) > 0 \}$ and $R_- = \{p \in \Sigma : \alpha(X_p) < 0 \}$. We note that a key property of the dividing set is that one can find an area form on $\Sigma$ and a vector field directing the characteristic foliation of $\Sigma$ so that it points out of $R_+$ and has $\pm$-divergence on $R_\pm$. Given any oriented singular foliation $\mathcal{F}$ on $\Sigma$ and a collection of curves $C$ on $\Sigma$, we say $C$ \dfn{divides} $\mathcal{F}$ if $C$ divides $\Sigma$ into two pieces $R_+$ and $R_-$ with the above properties. 

If $\Sigma$ has Legendrian boundary, then $\tb(\bd\Sigma) = -\frac12 |\Gamma_{\Sigma} \cap \bd\Sigma|$ and $\rot(\bd\Sigma) = \chi(R_+) - \chi(R_-)$. 

Suppose $\Sigma$ is a convex surface transverse to a contact vector field $X$. Let $\phi_t$ be the flow of $X$. Since $\xi$ is invariant under translations in the $t \in \mathbb{R}$ direction, for a small $\epsilon > 0$ we call $\phi_{[-\epsilon,\epsilon]}(\Sigma)$ an \dfn{$I$-invariant neighborhood of $\Sigma$}. If $\mathcal{F}$ is any oriented singular foliation divided by the dividing curves $\Gamma$ on $\Sigma$, then inside any $I$-invariant neighborhood of $\Sigma$, one may isotope $\Sigma$ through convex surfaces so that its characteristic foliation is given by $\mathcal{F}$. This is called \dfn{Giroux flexibility}. Moreover, since the characteristic foliation on a surface determines the contact structure in a neighborhood of the surfaces, we see that the dividing curves of a convex surface ``essentially" determine the contact structure in a neighborhood of the surface. 
In particular, we have the following tightness criterion.
  
\begin{theorem}[Giroux's criterion]\label{thm:criterion}
  Let $(M,\xi)$ be a contact $3$-manifold and $\Sigma$ be a compact convex surface (possibly with Legendrian boundary) in $(M,\xi)$. 
  \begin{itemize}
    \item Suppose $\Sigma = S^2$. An $I$-invariant neighborhood of $\Sigma$ is tight if and only if $|\Gamma_{\Sigma}|=1$. 
    \item Suppose $\Sigma \neq S^2$. An $I$-invariant neighborhood of $\Sigma$ is tight if and only if there are no contractible dividing curves on $\Sigma$.
  \end{itemize}
\end{theorem}

Let $(M,\xi)$ be a tight contact 3-manifold and $T$ a convex torus in $(M,\xi)$. Since $\xi$ is tight, dividing curves on $T$ must be homologically essential curves by Theorem~\ref{thm:criterion}. Once we fix a homological basis of $T \simeq \R^2/\Z^2$, we denote the slope of the dividing curves by $s(T)$, and call it the \dfn{dividing slope of T}. Also, since the dividing set divides $T$ into two regions $R_+$ and $R_-$, there should be an even number of dividing curves. See Figure~\ref{fig:torus} for example.  
\begin{figure}[htbp]{\scriptsize
  \begin{overpic}[tics=20]{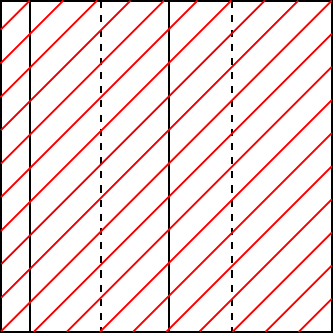}
  \end{overpic}}
  \caption{A characteristic foliation of a convex torus with two dividing curves. The dotted lines are dividing curves. The red lines are ruling curves and the black vertical lines are Legendrian divides.}
  \label{fig:torus}
\end{figure}
Using Giroux flexibility, we can $C^0$-isotop the surface $T$ so that the characteristic foliation on $T$ is in \dfn{standard form}. By this we mean that there are circles worth of singularities (these are points where $T\Sigma=\xi$) between two adjacent dividing curves, these are called \dfn{Legendrian divides}, and the rest of the characteristic foliation is made up of lines of any pre-chosen slope that is not equal to the dividing slope. These lines are called \dfn{ruling curves}. See Figure~\ref{fig:torus} for an example of a standard foliation. 

Now we will describe tight contact structures on a thickened torus $T^2 \times [0,1]$ and a solid torus $S^1 \times D^2$. To do so, we first review the Farey graph. First, we define the \dfn{Farey sum} of two rational numbers to be 
\[
  \frac{a}{b} \oplus \frac{c}{d} := \frac{a+c}{b+d}.  
\]
Also, we define \dfn{Farey multiplication} of two rational numbers to be 
\[
  \frac{a}{b} \bigcdot \frac{c}{d} := ad - bc.
\] 
Now consider the Poincar\'e disk in $\R^2$ equipped with the hyperbolic metric. Label the points $(0,1)$ by $0=0/1$ and $(0,-1)$ by $\infty=1/0$ and add a hyperbolic geodesic between the two points. Take the half circle with non-negative $x$ coordinate. Label any point halfway between two labeled points with the Farey sum of the two points and connect it to both points by a geodesic. Iterate this process until all the positive rational numbers are a label on some point on the half circle. Now do the same for the half circle with non-positive $x$ coordinate (for $\infty$, use the fraction $-1/0$). See Figure~\ref{fig:Farey}. We note that for two points on the Farey graph labeled $r$ and $s$, we have $|r \bigcdot s| = 1$ if and only if there is an edge between them. 

\begin{figure}[htbp]{\scriptsize
  \begin{overpic}[tics=20]{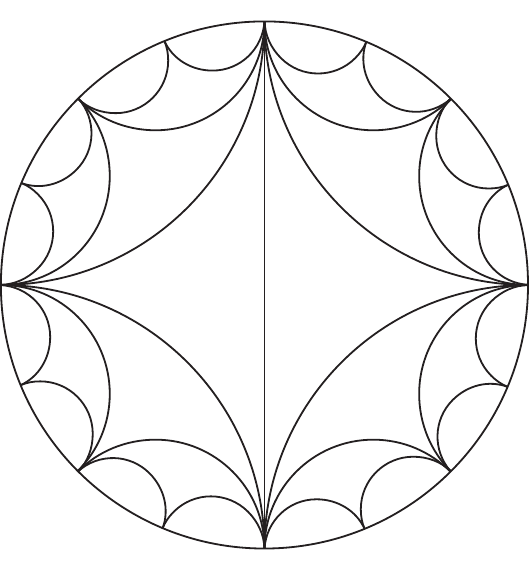}
    \put(123, 0){$\infty$}
    \put(125, 266){$0$}
    \put(-15, 132){$-1$}
    \put(257, 132){$1$}
    \put(20, 38){$-2$}
    \put(222, 38){$2$}
    \put(19, 232){$-1/2$}
    \put(218, 232){$1/2$}
    \put(60, 260){$-1/3$}
    \put(175, 260){$1/3$}
    \put(-13, 185){$-2/3$}
    \put(248, 185){$2/3$}
    \put(-17, 80){$-3/2$}
    \put(248, 80){$3/2$}
    \put(63, 8){$-3$}
    \put(175, 8){$3$}
  \end{overpic}}
  \caption{The Farey graph.}
  \label{fig:Farey}
\end{figure}

A \dfn{basic slice} $B_\pm(s,s')$ is defined to be a minimally twisting contact structure on $T^2 \times I$ such that $T^2\times\{0\}$ and $T^2\times \{1\}$ are convex tori with dividing slope $s$ and $s'$, respectively, where $s'$ is clockwise of $s$ and $|s \bigcdot s'| = 1$. Here \dfn{minimally twisting} means that any convex torus in $B_\pm(s,s')$ that is parallel to the boundary has a dividing slope that is clockwise of $s$ and anti-clockwise of $s'$ in the Farey graph. There are two non-isotopic contact structures satisfying such conditions and they differ by their coorientation. We denote one by $B_+(s,s')$ and the other by $B_-(s,s')$ and call them a positive and negative basic slice, respectively. Since $|s \bigcdot s'| = 1$, there is an edge between $s$ and $s'$ in the Farey graph. Thus we can describe a basic slice $B_\pm(s,s')$ as a decorated path $(s_0=s,s_1=s')$, consisting of a single edge between $s$ and $s'$, and the sign of the edge is the sign of the basic slice.     

We now consider a minimally twisting contact structure on $T^2\times [0,1]$ with dividing slopes $r$ on $T^2\times \{0\}$ and $s$ on $T^2\times \{1\}$.  Let $(r = s_0, \ldots, s_n = s)$ be the minimal path from $r$ to $s$ in the Farey graph. Then we can decompose a minimally twisting tight contact structure on $T^2\times [0,1]$ into basic slices
\[
 B(s_0,s_1) \cup \cdots \cup B(s_{n-1},s_n).
\]
So the contact structure is defined by a choice of signs on the edges between $s_{i-1}$ and $s_i$. If we have a non-minimal path going (with all vertices clockwise of $r$ and anti-clockwise of $s$, this will also define a contact structure on $T^2\times [0,1]$, but it will be tight if and only if it can be consistently shortened to a minimal path. We say the path can be consistently shortened if $|s_{i-1}\cdot s_{i+1}|=1$ and the edge from $s_{i-1}$ to $s_i$ and the edge from $s_i$ to $s_{i+1}$ are the same. The shortened path will be the result of removing $s_i$ and adding the edge from $s_{i-1}$ to $s_{i+1}$ with the sign of the removed edges.

It will be useful to have flexibility in the coordinate used on the boundary of a solid torus. To allow for this we describe a solid torus as follows. Consider $T^2\times [0,1]$ and choose some basis for the homology of $T^2$. The \dfn{solid torus with lower meridian of slope $r$} is formed from $T^2\times [0,1]$ by collapsing the leaves of a foliation of $T^2\times \{0\}$ by circles of slope $r$. We denote this solid torus $S_r$. We can similarly define the \dfn{solid torus with upper meridian of slope $r$} except we collapse leaves of the same foliation on $T^2\times \{1\}$ and denote the result $S^r$. 

We will now consider tight contact structures on a solid torus $S_r$ with boundary being convex with two dividing curves of slope $s$. We will denote such a contact structure by $S_r(s)$. Similarly $S^r(s)$ will denote a tight contact structure on a solid torus $S^r$ with convex boundary having two dividing curves of slope $s$. 

Kanda~\cite{Kanda:torus} showed that there exists a unique tight contact structure on $S_r(s)$ and $S^r(s)$ up to isotopy fixing boundary if $|s \bigcdot r| = 1$. Since there is an edge between $r$ and $s$ in the Farey graph, we can describe the tight contact structure on $S_r(s)$ as a decorated path $(s_0 = r, s_1 =s)$, consisting of a single edge between $r$ and $s$, and the sign of the edge is $\circ$. For $S^r(s)$, we use $(s_0=s, s_1=r)$ instead.              

Next, we described any tight contact structure on $S_r(s)$ and $S^r(s)$ for any $r,s \in \mathbb{Q}$. Let $(r = s_0, \ldots, s_n = s)$ be the minimal path from $r$ to $s$ in the Farey graph. Then we can decompose a tight contact structure on $S_r(s)$ into 
\[
  S_r(s) = S_r(s_1) \cup B(s_1,s_2) \cup \cdots \cup B(s_{n-1},s_n).
\] 
Thus we can describe a tight contact structure on $S_r(s)$ using a decorated path by assigning $\circ$ to the edge $(s_0,s_1)$ and assigning $+$ or $-$ to all other edges $(s_i,s_{i+1})$ according to the sign of the basic slice $B(s_i,s_{i+1})$ for $1 \leq i \leq n-1$ . For $S^r(s)$, we use the path $(s = s_0, \ldots, s_n = r)$ instead and it is the last edge in the path that is assigned a $\circ$ while the others have a $+$ or a $-$.  Giroux \cite{Giroux:classification} and Honda \cite{Honda:classification} proved that any tight contact structure on $S_r(s)$ or $S^r(s)$ can be described this way. (Note all of these contact structures are distinct, but we will not need to determine when two are the same.)

%-----------------------------------------------------------------
\subsection{Rational open book decompositions}\label{subsec:robd}
%-----------------------------------------------------------------
Here, we briefly review rational open book decompositions and their compatible contact structures. For more details, see \cite{BEV:robd}.

If $\mathcal{B}$ is a rationally null-homologous oriented link in a $3$-manifold $M$, then a \dfn{rational Seifert surface} for $\mathcal{B}$ is the image of a map $f:\Sigma\to M$ of an oriented surface $\Sigma$, such that $f$ is an embedding on the interior of $\Sigma$ and $\partial \Sigma$ maps to $\mathcal{B}$, the restriction of $f$ to each component of $\partial \Sigma$ is a positive cover of a component of $\mathcal{B}$, or in other words a positively oriented vector to $\partial \Sigma$ maps to a positively oriented vector on $\mathcal{B}$.

\begin{definition}\label{def:robd}
  A \dfn{rational open book decomposition} for a $3$-manifold $M$ is a pair $(\mathcal{B},\pi)$ consisting of an oriented link $\mathcal{B} = (K_1,\ldots,K_n)$ in $M$ and a fibration $\pi:(M \setminus \mathcal{B}) \to S^1$ such that for any $\theta \in S^1$, $\overline{\pi^{-1}(\theta)}$ is a rational Seifert surface for $\mathcal{B}$. We say $\mathcal{B}$ is the \dfn{binding} of the open book decomposition and each $\overline{\pi^{-1}(\theta)}$ is the \dfn{page} of the open book decomposition. A {\em rationally fibered knot} is one that is the binding of a rational open book. If $\overline{\pi^{-1}(\theta)}$ is a Seifert surface for $\mathcal{B}$ then we call $(\mathcal{B}, \pi)$ an open book decomposition, or sometimes an \dfn{honest} open book decomposition. 
\end{definition}

\begin{remark}
We note that in \cite{BEV:robd} the compatibility between the orientability of the rational Seifert surface and the binding was not made explicit, but was implicit throughout the paper. We make the necessary relation explicit in the definition above.  
\end{remark}

\begin{definition}\label{def:support}
  A (rational) open book decomposition $(\mathcal{B},\pi)$ \dfn{supports} a contact $3$-manifold $(M,\xi_{\mathcal{B}})$ if there exists a contact form $\alpha$ satisfying 
  \begin{itemize}
    \item $\ker\alpha$ is isotopic to $\xi_{\mathcal{B}}$, 
    \item $\alpha(v) > 0$ for any positively oriented tangent vector $v \in T\mathcal{B}$ and
    \item $d\alpha$ is a positive volume form of each page.
  \end{itemize} 
\end{definition}

The contact structure supported by a rational open book is unique up to isotopy \cite{BEV:robd, Grioux:openbook}. When talking about a fibered knot $K$ in a lens space, we will abuse notation and refer to a contact structure as supported by the knot $K$. In general, for multi-component links, one may have their complements fibered in many different ways and hence have many open book decompositions corresponding to the same binding, but if the complement of a knot fibers, then it fibers in a unique way. This may easily be seen by considering the Thurston norm \cite{Thurston1986} and noting that any incompressible surface in the same homology class as the fiber must be isotopic to the fiber. 

We note that the contact structure supported by an open book is very sensitive to the orientation of the binding $\mathcal{B}$, see \cite[Theorem~1.8]{BEV:robd}. However, in this paper, we only consider torus knots in lens spaces, and two equivalent torus knots with different orientations support contact structures that are contactomorphic, see \cite[Theorem~1.8.(2)]{BEV:robd}. Torus knots in lens spaces are rationally fibered, see \cite[Lemma~2.2]{BEV:robd}, and, as mentioned above, we will abuse notation (see for example Corollary~\ref{cor:signEquiv} and Propositions~\ref{prop:postive=univ}, and~\ref{prop:negative=ot}) by referring to {\em a contact structure supported by a torus knot in a lens space}, to mean the contact structure supported by the specific rational open book with binding a torus knot, on the lens space, as mentioned in the lemma. We will also use the phrase {\em lens space supported by a torus knot} to mean the same.

Let $K$ be a binding component of a (rational) open book decomposition $(\mathcal{B},\pi)$ of $M$ and let $N$ be a neighborhood of $K$.  Fix a reference framing on $K$. Then $\pi|_{M \setminus N}^{-1}(\theta) \cap \bd N$ is an essential simple closed curve on $\bd N$. We say the slope of this curve with respect to the reference framing is the \dfn{page slope} of $K$.  {The following theorem shows that for sufficiently small slopes,  the resulting contact structure of admissible transverse surgery on $K$ is supported by essentially the same open book.}

\begin{proposition}[Baker--Etnyre--Van-Horn-Morris \cite{BEV:robd}]\label{prop:robd-surgery}
  Suppose $K$ is a binding component of a (rational) open book decomposition $(\mathcal{B},\pi)$ supporting $(M,\xi_\mathcal{B})$. Then for any $r \in \mathbb{Q}$ less than the page slope of $K$, the resulting contact structure of admissible transverse $r$-surgery on $K$ is supported by $(\mathcal{B}^*,\pi)$ where $\mathcal{B}^*$ is the surgery dual of $\mathcal{B}$.  
\end{proposition}

The following lemma was proven in \cite{BaldwinEtnyre:transverse, EtnryeVelaVick:openbook} for honest open book decompositions but the same proof works for rational open book decompositions.

\begin{lemma}~\label{lem:binding}
  Let $(\mathcal{B},\pi)$ be a (rational) open book decomposition supporting $(M,\xi_{\mathcal{B}})$.  
  \begin{enumerate}
    \item A standard neighborhood $S_a$ of each binding component can be chosen so that $a$ becomes arbitrarily close to the page slope (which is measured by the framing induced from the polar coordinates of $S_a$).
    \item The complement of the binding is universally tight. Moreover, it does not contain Giroux torsion, but will remain tight when Giroux torsion is added. 
  \end{enumerate}
\end{lemma}

We finish this section by introducing a lemma which will be used in later sections.

\begin{lemma}\label{lem:coordinates}
  Let $(\mathcal{B},\pi)$ be a (rational) open book decomposition supporting $(M,\xi_{\mathcal{B}})$. Then there is a contact form $\alpha$ for $\xi_{\mathcal{B}}$ and polar coordinates $(r,\theta,\phi)$ near each binding component such that
  \[
    \alpha = \frac{1}{Ar^2+B}(r^2 d\theta + d\phi) 
  \]
  and the projection map is 
  \[
    \pi (r,\theta, \phi) = C\theta + D\phi
  \]
  for some constants $A,B,C,D$. The Reeb vector field for this contact form is given by 
  \[
     R_{\alpha} = A\frac{\bd}{\bd\theta} + B\frac{\bd}{\bd\phi}.
  \]
\end{lemma}

We note that in the proof of this lemma we will see that one has great flexibility in the exact form for $\alpha$, but the choice made in the lemma will give us the needed control over the Reeb vector field to prove Theorem~\ref{thm:cap} where we discuss attaching a symplectic handle to the binding of an open book. 

\begin{proof}
  Let $N$ be a small neighborhood of the binding $B$. By the definition, $\bd N \cap \pi^{-1}(c_0)$ is a set of homologically essential curves on $\bd N$ for $c_0 \in S^1$. Thus we can choose polar coordinates $(r,\theta,\phi)$ near a binding component such that
  \[
    \pi(r,\theta,\phi) = C\theta + D\phi
  \]
  for some constants $C$ and $D$.  

  Since each binding component is a transverse knot, and transverse knots have standard neighborhoods, we can write the contact form $\alpha$ near the binding in terms of the polar coordinates as follows: 
  \[
    \alpha = f(r,\theta, \phi)(r^2 d\theta + d\phi), 
  \]
  where $f$ is a positive function. 
  Since we are interested in a contact structure, not a form, we can scale the contact form $\alpha$ using any positive function. If we rescale by $\frac{1}{f(Ar^2+B)}$  for any positive constants $A$ and $B$ near a binding component, then we can rewrite $\alpha$ as follows:
  \[
    \alpha = \frac{1}{Ar^2+B}(r^2 d\theta + d\phi). 
  \]
  One may easily compute the Reeb vector field is the one claimed.
\end{proof}

%---------------------------------------------------------------------------------------------
\subsection{Torus knots in lens spaces}\label{subsec:torus}
%---------------------------------------------------------------------------------------------
For any relatively prime integers $r$ and $s$,  we define a lens space $L(r,s)$ by
\[
  L(r,s) = S^3_U({-r/s})
\]
where $U$ is the unknot in $S^3$ and $S^3_{U}(-r/s)$ denotes $-r/s$ Dehn surgery on $U$. Let $T$ be a Heegaard torus of $S^3$ which is the boundary of a neighborhood of $U$. Let $(\lambda_U, \mu_U)$ be the coordinates on $T$ where $\mu_U$ is a meridian and $\lambda_U$ is the Seifert longitude of $U$. By observing that $T$ is still a Heegaard torus of the lens space $L(r,s)$ obtained by surgery on $U$, we can keep using $(\lambda_U, \mu_U)$ as coordinates for the Heegaard torus in $L(r,s)$. 

We now set up some convenient terminology. For any relatively prime integers $p$ and $q$, we define a \emph{$(p,q)$--torus knot $T_{p,q}$ in $L(r,s)$} to be a simple closed curve on $T$ in the homology class $p\lambda_U+ q\mu_U$. Unlike in $S^3$, the notion of positive and negative torus knots is ambiguous in general lens spaces. Motivated by Lemma~\ref{lem:fibered}, the following definition was proposed in \cite{BEV:robd}: We say $T_{p,q}$ is a \dfn{positive} torus knot if $q/p \in (0,-r/s)$ when $-r/s > 0$, or $q/p \notin (-r/s,0)$ when $-r/s < 0$ (which is equivalent to $q/p$ being counterclockwise of $-r/s$ and clockwise of $0$ in the Farey graph, see Section~\ref{subsec:lens}), and a \dfn{negative} torus knot if $q/p \in (-r/s,0)$ when $-r/s < 0$, or $q/p \notin (0,-r/s)$ when $-r/s>0$ (which is equivalent to $q/p$ being clockwise of $-r/s$ and counterclockwise of $0$ in the Farey graph, see Section~\ref{subsec:lens}). We also say $T_{p,q}$ is \dfn{trivial} if $|pr+qs| = 1$ or $|q| = 1$, and otherwise, \dfn{nontrivial}. Notice that a trivial torus knot is isotopic to one of the cores of the Heegaard tori. Whenever a knot sits on a surface, the surface induces a framing on the knot and the framing for $T_{p,q}$ induced from the Heegaard torus $T$ is called the \dfn{torus framing} of $T_{p,q}$ and is denoted by $f_T$.

Baker, Van-Horn-Morris, and the first author showed that every torus knot in any lens space is rationally fibered.

\begin{lemma}[Baker--Etnyre--Van-Horn-Morris \cite{BEV:robd}]\label{lem:fibered}
  A torus knot $T_{p,q}$ in a lens space $L(r,s)$ is a rationally fibered knot. Moreover,
  \begin{itemize}
    \item the torus framing is larger than the page slope for positive torus knots,
    \item the torus framing is less than the page slope for negative torus knots,
  \end{itemize}   
  and if the torus knot is nontrivial, then the page and torus framings differ by more than one.
\end{lemma}
One may see why this is true (especially the last statement) by noting that $L(r,s)$ is the union of two solid tori $S_1$ and $S_2$ and that a (rational) Seifert surface for $T_{p,q}$ can be built by taking some number of meridional disks in $S_1$ and some other number in $S_2$ and then resolving their intersections on $\partial S_1=\partial S_2$ in either a positive or negative way. In the case $T_{p,q}$ is null-homologous, each resolution contributes $\pm 1$ to the difference between the torus and Seifert framing. This is similar when $T_{p,q}$ is only rationally null-homologous, except that the ``Seifert framing" is not exactly a framing. The fact that the complement of $T_{p,q}$ is fibered follows similarly to the same fact in $S^3$. Specifically, the complement is build by gluing $S^1\times [1,2]\times [0,1]$ to $S_1$ and $S_2$ so that $S^1\times \{i\}\times [0,1]$ is glued to an annulus in $\partial S_i$. The fibration of the $S_i$ by disks (thought of as $0$-handles) and $S^1\times [1,2]\times [0,1]$ by disks (thought of as $1$-handles) gives the fibration of the complement of $T_{p,q}$ by rational Seifert surfaces. 

%------------------------------------------------------------------
\subsection{Symplectic handle attachments}\label{subsec:handles}
%------------------------------------------------------------------
Here we review symplectic $2$-handles constructed by Gay \cite{Gay:handle}. First, for an $n$-dimensional $k$-handle $H = D^k \times D^{n-k}$, we denote the attaching region $\bd D^k \times D^{n-k}$ by $\bd_- H$ and denote $D^k \times \bd D^{n-k}$ by $\bd_+ H$. Also, we orient $\bd_- H$ as the opposite of the boundary orientation and orient $\bd_+ H$ as the usual boundary orientation. These are the orientations one would choose if they wanted to think of the $k$-handle as a (rel. boundary) cobordism from $\partial_-$ to $\partial_+$. We recall some other standard notation for handles. The disk $D^k\times \{0\}$ is called the \dfn{core} of the handle and the disk $\{0\}\times D^{n-k}$ is called the \dfn{co-core}. The boundary of the core is called the \dfn{attaching sphere} of the handle and the boundary of the co-core is called the \dfn{belt sphere} of the handle. 

\begin{definition}\label{def:handles}
  Let $H$ be a $4$-dimensional $k$-handle and $\omega_H$ be a symplectic structure on $H$. 
  \begin{itemize}
    \item $(H,\omega_H)$ is a \dfn{symplectic $k$-handle with convex boundaries} if there is a Liouville vector field on $H$ that points into $H$ along $\bd_- H$ and out of $H$ along $\bd_+ H$. We also call it a \dfn{convex $k$-handle} in short. 

    \item $(H,\omega_H)$ is a \dfn{symplectic $k$-handle with concave boundaries} if there is a Liouville vector field on $H$ that points out of $H$ along  $\bd_- H$ and into $H$ along $\bd_+ H$. We also call it a \dfn{concave $k$-handle} in short. 
  \end{itemize}
\end{definition}

We note that convex handles can be attached to a convex boundary to created a new symplectic manifold that also has a convex boundary, and similarly for concave handles. 
\begin{example}\label{ex:handles}
  Weinstein $k$-handles are convex $k$-handles for $k = 0, 1, 2$. By turning them upside down, we obtain concave $(4-k)$-handles.
\end{example}

Notice that we cannot build a closed symplectic manifold using just concave or convex handles. In \cite{Gay:handle}, Gay showed how to attach a $2$-handle to a convex boundary to create a symplectic manifold with concave boundary. We will describe his work below, but we will call his handle attachment a \dfn{convex-concave $2$-handle}. We note now that one can hope to build a closed symplectic manifold by using convex handles, then one convex-concave $2$-handle and then several concave handles. A main goal of this paper is to show that this strategy can indeed be carried out. 

\begin{definition}\label{def:decomposition}
  We say a symplectic $4$-manifold $(X,\omega)$ admits a \dfn{symplectic handlebody decomposition} if it admits a handlebody decomposition consisting of convex, concave, and convex-concave handles.
\end{definition}

The construction of Gay's convex-concave $2$-handles is somewhat similar to the Weinstein handle construction \cite{Weinstein:handle}, but a convex-concave handle requires a pair of dilating and contracting Liouville vector fields. The model handle is defined as a subset of $(\mathbb{R}^4, \omega_0)$, where $\omega_0 = r_1dr_1d\theta_1 + r_2dr_2d\theta_2$ is the standard symplectic structure with polar coordinates $(r_1,\theta_1,r_2,\theta_2)$. Consider the function $f(r_1,\theta_1, r_2,\theta_2)=-r_1^2+r^2_2$. For positive integers $C$ and $D$, the Liouville vector fields $X_-$ and $X_+$ are defined by 
\begin{align*}
  X_- &= \left(\frac{r_1}{2} - \frac{C}{r_1}\right) \frac{\bd}{\bd r_1} + \frac{r_2}{2} \frac{\bd}{\bd r_2},\\ 
  X_+ &= -\frac{r_1}{2} \frac{\bd}{\bd r_1}  - \left(\frac12 r_2 - \frac{D}{r_2}\right)\frac{\bd}{\bd r_2}.
\end{align*}
$X_-$ is a dilating Liouville vector field transverse to $f^{-1}(y)$ for $-2C < y<0$ and $X_+$ is a contracting Liouville vector field transverse to $f^{-1}(y)$ for $0<y < 2D$. Now we define the handle $H$ as a subset of $f^{-1}[-2C+\epsilon,2D-\epsilon]$ for some small $\epsilon >0$ such that $\bd_- H$ is a subset of $f^{-1}(-2C+\epsilon)$ and $\bd_+ H$ consists of a subset of $f^{-1}(2D - \epsilon)$ and a certain interpolation from $f^{-1}(-2C + \epsilon)$ to $f^{-1}(2D - \epsilon)$. See Figure~\ref{fig:handle}.  

Unlike Weinstein $2$-handles, the attaching sphere $\bd_-H\cap \{r_2=0\}$ is a transverse knot. Also, they can only be attached along a special transverse link to make the entire result concave. 

\begin{figure}[htbp]{
  \vspace{0.2cm}
  \begin{overpic}%[grid,tics=10]
  {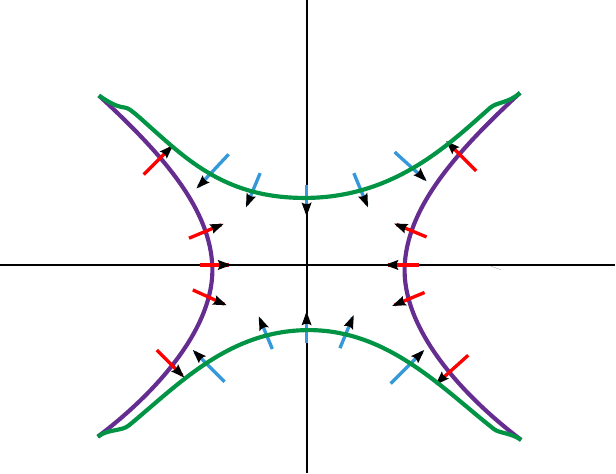}
   \put(155, 220){\large$r_2$}
   \put(285, 110){\large$r_1$}
   \put(212, 110){\large$X_-$}
      \put(122, 152){\large$X_+$}
  \end{overpic}}
  \vspace{0.2cm}
  \caption{A model for a convex-concave handle in $(\mathbb{R}^4,\omega_0)$. Here, $X_-$ is dilating as $\bd_-H$ is oriented as $-\partial H$, and $X_+$ is contracting as $\bd_+H$ is oriented as $\partial H$.}
  \label{fig:handle}
\end{figure}

\begin{definition}\label{def:nice}
  A transverse link $\mathcal{K} \subset (M,\xi)$ is \dfn{nicely fibered} if there exists a fibration $p\colon M \setminus \mathcal{K} \to S^1$ and a contact vector field $X$ on $M$ such that
  \begin{itemize}
    \item $X$ is transverse to the fibers of $p$
    \item For each binding component $K$, there are polar coordinates $(r,\theta,\phi)$ on a neighborhood of $K$ such that 
    \begin{itemize}
      \item $\bd / \bd r$ is tangent to the fibers.
      \item $dr(X) = 0$. 
      \item $X$ and $dp$ are both invariant under $\bd / \bd r$, $\bd / \bd \theta$ and $\bd / \bd \phi$.
    \end{itemize} 
    \item Let $X$ be a coorientation of both $\xi$ and the fibers $F$. The characteristic foliation on $F$ is oriented as the oriented intersection $\xi_p\cap T_pF$, see \cite[Chapter~3]{GuilleminPollack2010}.  Near each binding component $K$  the characteristic foliation points towards $K$.
  \end{itemize}
\end{definition}

Now let $\mathcal{K} = (K_1,K_2,\ldots,K_m)$ be a nicely fibered link in $(M,\xi)$ (equipped with a reference framing) and let $\mathbf{n} = (n_1,n_2,\ldots,n_m)$ be integers that are greater than the fiber slopes of each link component.

\begin{theorem}[Gay \cite{Gay:handle}]\label{thm:concavehandle} 
  With the notation above, let $(W,\omega)$ be a strong symplectic filling of $(M,\xi)$. Then there exist a strong symplectic cap $(W',\omega')$ containing $(W,\omega)$, obtained by attaching $n_i$-framed convex-concave $2$-handles $(H_i,\omega_i)$ along all $K_i$.
\end{theorem}

%----------------------------------------------------------------
\subsection{Lens spaces bounding rational homology balls}\label{subsec:rhb}
%----------------------------------------------------------------
Recall that $B_{p,q}$ is a smoothing of the cyclic quotient singularity of type $(p^2,pq-1)$, {\it i.e.} $B_{p,q}$ is a smoothing of the quotient of $B^4 \subset \mathbb{C}^2$ under the $\mathbb{Z}_{p^2}$-action generated by 
\[
  (z_1,z_2) \mapsto (e^{2\pi i/p^2}z_1, e^{2\pi(pq-1)i/p^2}z_2).
\] 
It is a rational homology ball and can be described by a handlebody diagram, as shown in Figure \ref{fig:Bpq}. For more details, see \cite{LM:balls, Wahl:singularities}.

\begin{figure}[htbp]{\scriptsize
  \vspace{0.2cm}
  \begin{overpic}%[width=.4\textwidth, tics=20]
  {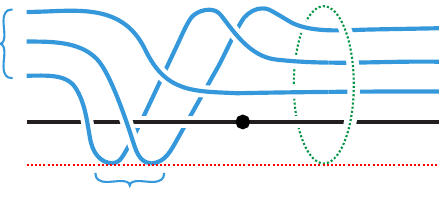}
   \put(55, -2){\color{lblue}$-p$}
   \put(-9, 72){\color{lblue}$q$}
   \put(180, 85){$-1$}
   
   \put(140, 93){$\mu_U$}
   \put(173, 6){$\lambda_U$}
  \end{overpic}}
  \vspace{0.2cm}
  \caption{The rational ball $B_{p,q}$ has boundary the lens space $L(p^2,pq-1)$. Taking the braid closure of the diagram in the figure will give a Kirby diagram for $B_{p,q}$.
  The red longitude is oriented left right and the green meridian is oriented clockwise.  The $-1$ framing on the $2$-handle is relative to the torus framing (note the blue curve sits on a Heegaard torus). }
  \label{fig:Bpq}
\end{figure}

\begin{lemma}\label{lem:Bpq}
  The boundary of $B_{p,q}$ is the lens space $L(p^2,pq-1)$. 
\end{lemma}

There are several standard arguments for proving this; we will present a proof which uses a method we will rely on heavily later in the paper.  We will require the following classical lemma, see for instance \cite[Lemma~9.I.4]{Rolfsen:book}. 

\begin{lemma}\label{lem:surfacetwist}
  Let $M$ be a 3-manifold with a Heegaard splitting $V_1\sqcup_{\phi} V_2$ where $V_i$, are genus $g$ handlebodies that are glued together by a diffeomorphisms $\phi:\bd V_1\to-\bd V_2$. Let $\gamma$ be a simple closed curve on $\bd V_1$, and suppose $M'$ is obtained by Dehn surgering $M$ along $\gamma$ with framing $f_{\bd V_1}\pm 1$. Then $M'$ has a Heegaard splitting $V_1\sqcup_{\tau^{\mp 1}_\gamma\circ\phi} V_2$, where $\tau_\gamma$ denotes the right-handed Dehn twist of $\bd V_1$ along $\gamma$. 
\end{lemma}

\begin{proof}[Proof of Lemma \ref{lem:Bpq}]
  Let $\mu_U$ and $\lambda_U$ be the two oriented curves in the boundary $\partial(B_{p,q})$ shown in Figure~\ref{fig:Bpq}. If we ignore the blue curve in the figure, then the boundary manifold is $S^1\times S^2$, thought of as $0$ surgery on the dotted curve in the figure, and is the union of two solid tori. Now $S^1\times S^2$ is the union of two solid tori. The first, $V_1$, is the surgery solid torus obtained from surgery on the dotted curve and the second, $V_2$, is its complement. Notice that $\lambda_U$ bounds a disk in the outside solid torus. We can think of $\partial(B_{p,q})$ as the result of surgery along the blue knot in $S^1\times S^2$. Since the blue knot lies on a Heegaard torus for $S^1\times S^2$ and we surger the blue knot with 1 less than the Heegaard torus framing, performing the blue surgery has the effect of modifying the gluing map of the Heegaard splitting as indicated in Lemma~\ref{lem:surfacetwist}. So we see that $\partial(B_{p,q})$ is a lens space.

  To compute which lens space, we will determine which simple closed curve in $\partial V_2$ bounds a disk in the inside solid torus. Observe that the 0-framed push-off of the black knot bounds a disk in $V_1$, this is the curve $\lambda_U$ thought of as sitting on a torus just inside the blue curve. We will now isotope $\lambda_U$ ``past'' the blue surgered curve, and identify it in $\partial V_2$. Pushing past the blue curve yields $\lambda_U-p(-p\mu_U+q\lambda_U)=p^2\mu_U+(1-pq)\lambda_U$. Thus we conclude that $\partial(B_{p,q})$ is the lens space $L(p^2,pq-1)$. 
\end{proof}

\begin{remark}\label{rem:pushingpast}
  \noindent We collect a few facts we will require about $B_{p,q}$. 
  \begin{enumerate}
  \item The rational homology ball $B_{p,q}$ admits Weinstein structures for each $\pm\xi_{std}$. See Figure~\ref{fig:BpqWeinstein} for Weinstein handlebody diagrams.

  \begin{figure}[htbp]{
    \vspace{0.2cm}
    \begin{overpic}[tics=20]{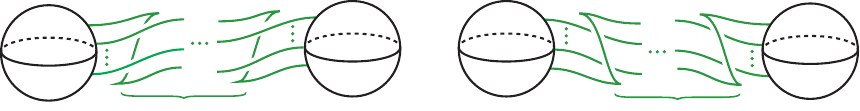}
    \put(86, -5){\color{dgreen}\small$p$}
    % \put(127, 27){\small$q$}
    \put(324, -5){\color{dgreen}\small$p$}
    %\put(349, 25){\small$q$}
    \end{overpic}}
    \vspace{0.2cm}
    \caption{Weinstein $B_{p,q}$ for $(L(p^2,pq-1),\xi_{std})$ and $(L(p^2,pq-1),-\xi_{std})$, there are $q$ strands. }
    \label{fig:BpqWeinstein}
  \end{figure}

  \item The contactomorphism $\tau\colon L(p^2,pq-1) \to L(p^2,pq-1)$ from Corollary~\ref{cor:signEquiv} below extends to a diffeomorphism of $\overline{\tau}\colon B_{p,q} \to B_{p,q}$.  This follows immediately from the definition of $\tau$.   

  \item Because we will make use of such computations later, we also compute for reference the image of the meridian $\mu_U$ of black curve in Figure~\ref{fig:Bpq} when we push it past the blue surgery; here we see $\mu_U \mapsto (1+pq)\mu_U-q^2\lambda_U$. So if we take $(\lambda_U,\mu_U)$ to be a basis for $H_1(T^2)$, then the matrix describing the images of the longitude and meridian of the black curve after isotoping past the blue surgery curve is 
  \[
    \begin{pmatrix}
      1-pq  & -q^2\\ 
      p^2 & 1+pq 
    \end{pmatrix}. 
  \]
  \end{enumerate}
\end{remark}

%---------------------------------------------------------------------------------------------
\subsection{Contact structures on lens spaces and decorated paths}\label{subsec:lens}
%---------------------------------------------------------------------------------------------
Here, we review how to build a contact structure on a lens space. Recall that in Section~\ref{subsec:convexBasic}, we use edges in the Farey graph to describe contact structures on solid tori $S_r(s)$ and $S^u(s)$. By gluing $S_{-r/s}(p/q)$ and $S^0(p/q)$ together along the boundary, we can construct a contact structure on $L(r,s)$ and describe it using a decorated path $P = \{s_0=-r/s, \ldots,s_n=p/q,\ldots, s_m=0\}$ in the Farey graph where the first and last edges of the path are decorated with $\circ$ and the rest by a $+$ or $-$. 

Giroux \cite{Giroux:classification} and Honda \cite{Honda:classification} classified tight contact structures on lens spaces. We can state part of their classification in terms of decorated paths in the Farey graph.

\begin{theorem}[Giroux \cite{Giroux:classification}, Honda \cite{Honda:classification}]\label{thm:lens}
  Let $P$ be a decorated path from $-r/s$ to $0$ in the Farey graph and $\xi$ be the corresponding contact structure on $L(r,s)$. Then
  \begin{enumerate}
    \item if $P$ is minimal, then $\xi$ is tight.
    \item if $P$ is minimal and all edges have the same sign except for the first and the last ones which are decorated with $\circ$, then $\xi$ is universally tight. 
    \item if $P$ is minimal and contains both $+$ and $-$ signs, then $\xi$ is tight but virtually overtwisted.  
    \item if $P$ is not minimal and there are two adjacent edges $(s_{i-1},s_{i})$ and $(s_i,s_{i+1})$ with different signs that can be shortened ({\it i.e.\ } $|s_{i-1}\bigcdot s_{i+1}|=1$), then $\xi$ is overtwisted.
  \end{enumerate}
\end{theorem}

See Figure~\ref{fig:decoratedFarey} for an example of an overtwisted contact structure on $L(3,1)$ obtained by gluing $S_{-3}(-8/5)$ and $S^0(-8/5)$ together.  

According to the previous theorem, there exist two universally tight contact structures on $L(p,q)$ and they differ by coorientation. We denote them by $\xi_{std}$ and $-\xi_{std}$, and call them the \dfn{standard contact structures} on $L(p,q)$.

\begin{remark}\label{rmk:ambiguity}
  In this paper, we will not strictly distinguish $\xi_{std}$ and $-\xi_{std}$, since they are contactomorphic (see Corollary~\ref{cor:signEquiv}). When we refer to a standard contact structure on $L(p,q)$, it could be either $\xi_{std}$ or $-\xi_{std}$.
\end{remark}

In later sections, we will glue two symplectic manifolds together along their boundaries, which are a lens space, using a contactomorphism. Thus we review contactomorphisms on lens spaces here. We denote by $\Cont(L(p,q), \xi_{std})$ the group of coorientation preserving contactomorphisms of $L(p,q)$ with its standard contact structure.  

\begin{theorem}[Min \cite{Min:cmcg}]\label{thm:cmcg}
  The contact mapping class group of $(L(p,q),\xi_{std})$ is 
  \[
    \pi_0(\Cont(L(p,q), \xi_{std})) = \begin{cases} 
      \mathbb{Z}_2 &\; q \not\equiv \pm1 \text{ and }\ q^2 \equiv 1 \modp p,\\
      1 &\; \text{otherwise}. \end{cases}
  \]
\end{theorem}

The following is a direct corollary of Theorem~\ref{thm:cmcg}.

\begin{corollary}\label{cor:signEquiv} Let $\xi_{std}$ be a standard contact structure on $L(p^2,pq-1)$.
  \begin{enumerate}
    \item The identity map  
    \[
      id\colon (L(p^2,pq-1), \xi_{std}) \to (L(p^2,pq-1), \xi_{std})
    \] 
    is a unique coorientation preserving contactomorphism up to contact isotopy.
    \item There exists a unique coorientation reversing contactomorphism 
    \[
      \tau\colon (L(p^2,pq-1), \xi_{std}) \to (L(p^2,pq-1), -\xi_{std})
    \] 
    up to contact isotopy. The diffeomorphism $\tau$ is defined on each of the Heegaard tori $S^1\times D^2$ as $(\theta, z)\mapsto (-\theta, \overline{z})$ where we think of $D^2$ as the unit disk in $\C$.
  \end{enumerate}
\end{corollary}

\begin{proof}
  Suppose $|p|>2$. Then, $(pq-1)^2 \not\equiv 1 \pmod {p^2}$, so $\pi_0(\Cont(L(p^2,pq-1),\xi_{std})) = 1$. If $|p|=2$, then for any odd number $q$, $L(p^2,pq-1) \cong L(4,1)$ and $\pi_0(\Cont(L(4,1),\xi_{std})) = 1$. This completes the proof of the first statement.
  
  It is well known ({\it c.f. \cite{GO:lens, Min:cmcg}}) that for any standard contact structure on a lens space, there exists a coorientation reversing contactomorphism $\tau$. Suppose there is another coorientation reversing contactomorphism $\tau'$. Then $\tau \circ (\tau')^{-1}$ is a coorientation preserving contactomorphism, and by the above argument, it is contact isotopic to the identity. Therefore, $\tau$ is contact isotopic to $\tau'$ and this completes the proof of the second statement. 
\end{proof}

%%%%%%%%%%%%%%%%%%%%%%%%%%%%%%%%%%%%%%%%%%%%%%%%%%%%%%%%%%%%%%%%%%%%%%%%%%%%%%%%%%%%%%%%%%%%%%%%%%%%%%%%%%%%%
\section{Handlebody construction of closed symplectic \texorpdfstring{$4$}{4}-manifolds}\label{sec:geometry}
%%%%%%%%%%%%%%%%%%%%%%%%%%%%%%%%%%%%%%%%%%%%%%%%%%%%%%%%%%%%%%%%%%%%%%%%%%%%%%%%%%%%%%%%%%%%%%%%%%%%%%%%%%%%%

In this section, we will construct closed symplectic $4$-manifolds with $b_2=1$ into which we can embed three of the $B_{p,q}$, and will show that they admit symplectic handlebody decompositions. In Section~\ref{subsec:cap}, we slightly modify Theorem~\ref{thm:concavehandle} to create a symplectic cap and in Section~\ref{subsec:nonloose}, we study certain surgeries on torus knots in some contact lens spaces, which enable us to understand what contact structures we constructed a cap for in Theorem~\ref{thm:cap}. Finally, in Section~\ref{subsec:pants}, we construct closed symplectic $4$-manifolds for each Markov triple. 

%-------------------------------------------------------------------
\subsection{Construction of a small symplectic cap} \label{subsec:cap}
%-------------------------------------------------------------------
Let $(\mathcal{B},\pi)$ be a (rational) open book decomposition supporting $(M,\xi_{\mathcal{B}})$. Suppose $\mathcal{B} = (K_1,\ldots,K_m)$ is a (reference framed) link, and $\mathbf{n} = (n_1,\ldots,n_m)$ is a set of integers that are greater than the page slopes of each binding component. Let $(\overline{\mathcal{B}},\overline{\pi})$ be the mirror of $(\mathcal{B},\pi)$ supporting $(-M,\xi_{\overline{\mathcal{B}}})$. Here by the mirror of $\mathcal{B}$, we just mean $\mathcal{B}$ thought of as sitting in $M$ with its reversed orientation and $\overline{\pi}$ is the obvious projection $(-M\setminus \overline{\mathcal{B}})\to S^1$.  (Note we are not reversing the orientation on $\mathcal{B}$.) We denote the result of admissible transverse $-n_i$-surgeries on every binding component by $(-M_{\overline{\mathcal{B}}}(-\mathbf{n}),\xi_{\overline{\mathcal{B}}}(-\mathbf{n}))$. Also recall that when we say $(C,\omega)$ is a symplectic cap for $(M,\xi)$, the orientation of $M$ is opposite of the usual boundary orientation of $C$.

\begin{theorem}\label{thm:cap}
  With the notations above, the cobordism $W$ from $M$ to $M_{\mathcal{B}}(\mathbf{n})$ obtained by attaching $n_i$-framed $2$-handles to $[0,1]\times M$ along on every $\{1\}\times K_i$ admits a symplectic structure $\omega$ that gives a strong symplectic cap for $(M,\xi_{\mathcal{B}}) \sqcup (-M_{\overline{\mathcal{B}}}(-\mathbf{n}),\xi_{\overline{\mathcal{B}}}(-\mathbf{n}))$.  
\end{theorem}

\begin{proof}
  We will first show that $\mathcal{B}$ is a nicely fibered link in $(M,\xi_{\mathcal{B}})$. First, take a contact form $\alpha$ of $\xi_{\mathcal{B}}$ and polar coordinates $(r,\theta,\phi)$ near each binding component from Lemma~\ref{lem:coordinates}. Then we choose the Reeb vector field $R_{\alpha}$ as our contact vector field in Definition~\ref{def:nice} and directly calculate it to be
  \[
    R_{\alpha} = A\frac{\bd}{\bd\theta} + B\frac{\bd}{\bd\phi}.
  \]
  Since $d\alpha$ is a volume form of each page and $i_{R_\alpha}d\alpha = 0$, clearly $R_{\alpha}$ is transverse to each page, which verifies the first bullet point of Definition~\ref{def:nice}. Also it is straightforward to verify the statements in the second bullet point. For the third bullet point, as in the proof of Lemma~\ref{lem:coordinates}, the characteristic foliation of a page is defined by the Liouville vector field so it points in towards the binding components.   

  Thus we can apply Theorem~\ref{thm:concavehandle} to $([0,1] \times M, d(e^t\alpha))$, a piece of the symplectization of $(M,\xi)$, to obtain a symplectic cap with two boundary components. We are now left to show that the contact manifold obtained on the upper boundary of the handle attachment is $(-M_{\overline{\mathcal{B}}}(-\mathbf{n}),\xi_{\overline{\mathcal{B}}}(-\mathbf{n}))$. Clearly the handle attachment gives a cobordism from $M$ to $M_{\mathcal{B}}(\mathbf{n})$. Since we consider concave boundaries, we reverse the orientation of $M_{\mathcal{B}}(\mathbf{n})$ and obtain $-M_{\overline{\mathcal{B}}}(-\mathbf{n})$. Let $\mathcal{B}^*$ be the surgery dual link of $\mathcal{B}$. The proof of \cite[Theorem~1.2]{Gay:handle} and \cite[Addendum 5.1]{Gay:handle} implies that the resulting contact structure is supported by $\overline{\mathcal{B}^*}$. According to the uniqueness of a contact structure supported by a rational open book decomposition (\cite[Theorem~1.7]{BEV:robd}) and Proposition~\ref{prop:robd-surgery}, it is contactomorphic to $\xi_{\overline{\mathcal{B}}}(-\mathbf{n})$.
  \end{proof}

%----------------------------------------------------------------------------
\subsection{Non-loose torus knots in lens spaces}\label{subsec:nonloose}
%----------------------------------------------------------------------------
To utilize Theorem~\ref{thm:cap}, we need to identify the resulting contact structure on $-M_{\overline{\mathcal{B}}}(-\mathbf{n})$, the manifold obtained by admissible transverse surgery along the binding of a (rational) open book decomposition. In general, it is not easy to describe this new contact structure in a more standard form, especially when the contact structure $\xi_{\overline{\mathcal{B}}}$ is overtwisted (which it frequently will be if $\xi_{\mathcal{B}}$ is tight). The main result of this section is Theorem~\ref{thm:nonloose-surgery}, which makes such an identification in a special case.

We first characterize the contact structures on a lens space supported by a torus knot. Let $L(r,s)$ be a lens space for relatively prime integers $r,s$. Recall from Section~\ref{subsec:torus} that a torus knot $T_{p,q}$ in $L(r,s)$ is a \dfn{positive} torus knot if $q/p$ is clockwise of $0$ and counterclockwise of $-r/s$ in the Farey graph, and a \dfn{negative} torus knot if $q/p$ is clockwise of $-r/s$ and counterclockwise of $0$ in the Farey graph. The following proposition is straightforward from \cite[Theorem~1.8]{BEV:robd}.

\begin{proposition}\label{prop:postive=univ}
  Let $T_{p,q}$ be a positive torus knot in $L(r,s)$ and $\xi_{p,q}$ be the contact structure on $L(r,s)$ supported by $T_{p,q}$. Then $\xi_{p,q}$ is a standard contact structure on $L(r,s)$. 
\end{proposition}

Next we consider negative torus knots. To do so, we require a few preliminary results. 

\begin{definition}\label{def:nonloose}
  Suppose $(M,\xi)$ is an overtwisted contact $3$-manifold. A Legendrian knot $L \subset (M,\xi)$ is called \dfn{non-loose} if $(M \setminus N, \xi|_{M \setminus N})$ is tight where $N$ is a standard neighborhood of $L$. Similarly,  a transverse knot $K \subset (M,\xi)$ is called \dfn{non-loose} if $(M \setminus K, \xi|_{M\setminus K})$ is tight.   
\end{definition}

In \cite{EMM:nonloose}, non-loose torus knots in any contact structure on $S^3$ were classified. We can adapt the same method to describe the contact structure on a lens space supported by a negative torus knot in terms of decorated paths in the Farey graph.

Let $P$ be a decorated path in the Farey graph for $(L(r,s),\xi)$. There are two important decorated paths we need to consider: \dfn{consistent paths} and \dfn{totally inconsistent paths}. A \dfn{consistent path} $P$ is a decorated path $\{s_0,\ldots,s_n\}$ where the signs of all edges are identical except for the first and the last ones. The signs of the two edges $(s_0,s_1)$ and $(s_{n-1},s_n)$ are $\circ$. A \dfn{totally inconsistent path at $p/q$} is a decorated path $P$ where all signs of the edges $(s,s')$ clockwise of $p/q$ are positive (resp.\ negative) and all signs of the edges $(s,s')$ counterclockwise of $p/q$ are negative (resp.\ positive) except for the first and last ones. The signs of the two edges are $\circ$. See Figure~\ref{fig:decoratedFarey} for a totally inconsistent path at $-8/5$.  

\begin{figure}[htbp]{\scriptsize
  \vspace{0.2cm}
  \begin{overpic}%[grid,tics=10] 
  {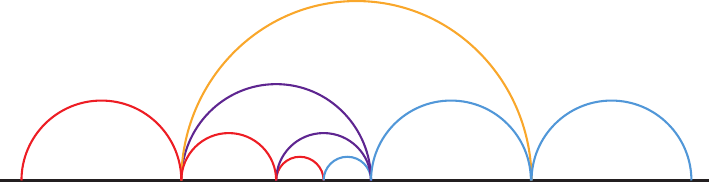}
       \put(45, 16){\Large \color{nred} $\circ$}
     \put(105, 8){\Large \color{nred} $+$}
     \put(140, 5){\large \color{nred} $+$}
       \put(291, 16){\Large \color{nblue} $\circ$}
     \put(210, 14){\Large \color{nblue} $-$}
     \put(163, 5){\large \color{nblue} $-$}
   \put(3, -7){$-3$}
   \put(80, -7){$-2$}
   \put(125, -7){$-\frac{5}{3}$}
   \put(148, -7){$-\frac{8}{5}$}
   \put(171, -7){$-\frac{3}{2}$}
   \put(249, -7){$-1$}
   \put(331, -7){$0$}
  \end{overpic}}
  \vspace{0.2cm}
  \caption{A decorated path in the Farey graph for the contact structure on $L(3,1)$ supported by the rational open book decomposition $(T_{8,-5},\pi)$.}
  \label{fig:decoratedFarey}
\end{figure}

Let $T_{p,q}$ be a torus knot in $L(r,s)$. Recall from Section~\ref{subsec:torus} that $T_{p,q}$ is a simple closed curve on a Heegaard torus of $L(r,s)$. Thus there is a framing for $T_{p,q}$ induced from the Heegaard torus. We call it the \dfn{torus framing} of $T_{p,q}$. Also recall from Section~\ref{subsec:torus} that $T_{p,q}$ is \dfn{trivial} if $|pr+qs| = 1$ or $|p| = 1$, and \dfn{nontrivial} otherwise. 

\begin{proposition}\label{prop:negative=ot}
  Suppose $\xi_{p,q}$ is the contact structure on $L(r,s)$ supported by a negative torus knot $T_{p,q}$. Then it is overtwisted if $T_{p,q}$ is nontrivial. Further,  $\xi_{p,q}$ may be described by a totally inconsistent path for $L(r,s)$ at $q/p$. 
\end{proposition}

\begin{proof}
  Let $N$ be a neighborhood of $T_{p,q}$ and define $C = L(r,s) \setminus N$. First we consider $N$ as a standard neighborhood $S_a$ of $T_{p,q}$ in $\xi_{p,q}$ for some $a \in \mathbb{Q}$ and let $\xi_C$ be the restriction of $\xi_{p,q}$ to $C$. By Lemma~\ref{lem:binding}, $\xi_C$ is universally tight and we may assume that $a \in \mathbb{Q}$ is any slope less than the page slope. The torus knot $T_{p,q}$ sits on a Heegaard torus $T$ of $L(r,s)$, and since $T_{p,q}$ is a negative torus knot, Lemma~\ref{lem:fibered} tells us that the torus framing is less than the page slope so we can assume $a$ is the slope corresponding to the torus framing. Now $N$ is a neighborhood of a Legendrian knot $L$ that is a Legendrian approximation of $T_{p,q}$, see \cite[Section~2]{EtnyreHonda01}, and its contact framing  agrees with the torus framing. We can smoothly perturb $T$ so that $L$ lies on $T$ and perturb $T$ again while fixing $L$ so that it becomes convex with dividing slope $q/p$. Notice that $L$ becomes a Legendrian divide of $T$, 
  
  which is the trace of singular points of the characteristic foliation of $T$. See Figure \ref{fig:torus} for example.  Now $T$ splits $L(r,s)$ into two solid tori $S_1$ and $S_2$ and hence the path describing the contact structure $\xi_{p,q}$ into two paths $P_1$ and $P_2$. Since the core of both $S_1$ and $S_2$ are homotopically nontrivial in $C$, $S_1$ and $S_2$ unwrap in coverings of $C$. Thus the contact structures restricted to $S_1$ and $S_2$ should both be universally tight. This implies that each path only contains a single sign. Now there are two cases to consider. First, both paths $P_1$ and $P_2$  have the same sign. In this case, the decorated path for $L(r,s)$ is a totally consistent path. In \cite[Lemma~3.16]{EMM:nonloose}, it was shown that adding Giroux torsion to $C$ in the totally consistent setting always produces an overtwisted contact structure. (In \cite{EMM:nonloose} only the case of $S^3$ was considered, but the proof only used a thickened torus containing $T_{p,q}$ and so applies to lens spaces as well.)  By Lemma~\ref{lem:binding}, adding Giroux torsion to the complement of $T_{p,q}$ in $\xi_{p,q}$ gives a tight contact manifold. Thus the signs in $P_1$ and $P_2$ must be opposite and the path describing $\xi_{p,q}$ is totally inconsistent at $p/q$. 
  
  Since $T_{p,q}$ is nontrivial, there is an edge between $(q/p)^a$ and $(q/p)^c$ in the Farey graph by \cite[Lemma~2.10]{EMM:nonloose}. Thus the decorated path is not minimal and we can shorten the path by merging two edges $((q/p)^a,q/p)$ and $(q/p,(q/p)^c)$. (See Section~\ref{subsec:lens} for notation.) However, since the two edges had different signs, $\xi_{p,q}$ is overtwisted by Theorem~\ref{thm:lens}. 
\end{proof}

Now we return to the contact manifolds obtained by surgery on torus knots. We first study a certain surgery on $T^2 \times I$, for which we require the following amusing lemma. Before stating the lemma we recall that Dehn filling a $3$-manifold $Y$ with a torus $T$ in its boundary is the result of gluing a solid torus to $Y$ along $T$. The filling is determined by the curve on $T$ to which the meridian to the solid torus is sent by the gluing. 

\begin{lemma}\label{lem:whysomanypants}
  Let $P$ be a pair of pants.  Dehn filling any boundary component of $P \times S^1$ along the curve $\{p\} \times S^1$ results in a connected sum of two solid tori, and this homeomorphism sends $\{p\} \times S^1$ to the meridians of the solid tori. 
\end{lemma}

\begin{proof}
  Let $T$ be the boundary component of $P \times S^1$ that is filled and $\gamma$ be an essential arc on $P \times \{x\}$ such that $\partial\gamma \subset T$ (Here by an ``essential" arc we mean that it is not isotopic into the boundary of $P$). See Figure~\ref{fig:essential}. 
  \begin{figure}[htbp]{\scriptsize
  \vspace{0.2cm}
  \begin{overpic}%[grid,tics=10] 
  {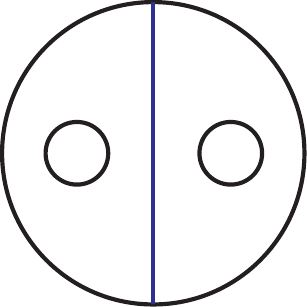}
    \put(60, 22){\large$\gamma$}
  \end{overpic}}
  \vspace{0.3cm}
  \caption{An essential arc on a pair of pants}
  \label{fig:essential}
\end{figure}
  Then $A = \gamma \times S^1$ is an essential annulus in $P \times S^1$, and each boundary component bounds a disk after the Dehn filling. (Here by ``essential" annulus we mean that it is not isotopic to an annulus in the boundary of the $3$-manifold.) The union of the two disks and $A$ is an essential sphere. We cut the manifold along this sphere and separated it into two components. Each component is homeomorphic to $(D^2 \times S^1)\setminus B^3$. To see this, we note that each component $C$ is made up of two pieces. One coming from $P\times S^1$ cut along the annulus is $T^2\times I$ (as it is an annulus times $S^1$), and the other piece coming from the complement of two meridional disks in the Dehn filling torus is $D^2\times I$. Notice that $\partial C$ consists of $T^2$ and $S^2$. If we glue a ball to $S^2$ we will see the result of Dehn filling one boundary component of $T^2\times I$. This is a solid torus. So, our component $C$ is a solid torus with a ball removed.  
\end{proof}

We now consider a similar situation in the contact geometry setting, but first, we recall from Section~\ref{subsec:convexBasic}, that any decorated path in the Farey graph can be used to define a contact structure (possibly overtwisted) on $T^2\times[0,1]$ by stacking basic slices. In addition, that contact structure is tight if and only if it can be consistently shortened to a minimal decorated path. 
\begin{lemma}\label{lem:pants-tori}
  Let $K$ be a slope $0$ curve in $T^2\times \{0\}$ in the contact structure on $T^2\times [-1,1]$ given by the union of basic slices $B_\pm(-1,0) \cup B_\mp(0,\infty)$. There is a non-loose Legendrian representative $L$ of $K$ such that the contact framing is $1$ larger than the torus framing. Moreover, Legendrian surgery on $L$ yields a connected sum of two tight solid tori $S^0(-1) \,\#\, S_0(\infty)$.
\end{lemma}

\begin{proof}
  We first show that the claimed $L$ exists. The union $B_\pm(-1,0) \cup B_\mp(0,1)$ is called a length $2$ balanced continued fraction block with central slope $0$. In \cite[Theorem~1.11]{ChakrabortyEtnyreMin20Pre} it was shown that inside such a thickened torus there is a unique Legendrian knot $L$ isotopic to the slope $0$ curve with framing one larger than the torus framing. Below we will see Legendrian surgery on $L$ yields a tight contact structure and thus $L$ is non-loose.

  Since Legendrian surgery on $L$ is a topological $0$-surgery on $K$, Lemma \ref{lem:whysomanypants} shows that the resulting manifold is a connected sum of two solid tori such that the $0$-slope curves on $T^2\times \{-1,1\}$ become the meridians. It remains to show the contact structure on each component is tight. Let $M = B_\pm(-1,0) \cup B_\mp(0,\infty) \setminus N$ where $N$ is a standard neighborhood of $L$. Take a properly embedded essential annulus $A$ in $M$ where each boundary component of $A$ is a $0$-slope curve on $\bd N$. We cut $M$ along $A$ and edge round, we obtain tight on two copies of $T^2 \times I$. Indeed, the proof of Theorem~1.10 in \cite{ChakrabortyEtnyreMin20Pre} shows that $M$ cut along $A$ is $B(-1,-1)$ and  $B_\mp(1,\infty)$, where $B(-1,-1)$ is an $I$-invariant contact structure on $T^2\times[0,1]$.
  
  Now, recall that Legendrian surgery on $L$ in $B_\pm(-1,0) \cup B_\mp(0,\infty)$ is the result of removing a neighborhood of $L$ and gluing in a solid torus with a tight contact structure. When we remove the neighborhood of $L$ the resulting manifold is a pair-of-pants times $S^1$, and we are Dehn filling one of its boundary components. So smoothly, we are in the situation of Lemma~\ref{lem:whysomanypants}. Thus, we know that the result of Dehn filling $P\times S^1$ is the connected sum of two solid tori formed by Dehn filling the complements of the annulus in $P\times S^1$. Above we saw that the complement of this annulus in our case are the $B(-1,-1)$ and  $B_\mp(1,\infty)$. We are Dehn filling each of these with a solid torus with meridional slope $0$. So according to the classification of contact structures on solid tori dsicussed in Sectxion~\ref{subsec:convexBasic} we see that the result of Legendrian surgery should be $S^0(-1) \,\#\, S_0(\infty)$. 
\end{proof}

Let $T_{p,q}$ be a torus knot in $L(r,s)$ and $\xi_{p,q}$ be the contact structure on $L(r,s)$ supported by $T_{p,q}$. When we perform a surgery on $T_{p,q}$, we call it a \dfn{torus framing surgery} if the surgery coefficient is the torus framing. Recall from Section~\ref{subsec:lens} that $q'/p' = (q/p)^c$ is the largest rational number satisfying $pq' - p'q = -1$ and $q''/p'' = (p/p)^a$ is the smallest rational number satisfying $pq'' - p''q = 1$.

\begin{theorem}\label{thm:nonloose-surgery}
  The torus framing admissible transverse surgery on a nontrivial negative torus knot $T_{p,q}$ in $(L(r,s), \xi_{p,q})$ results in a connected sum of standard contact structures on $L(p,-q)$ and $L(ps+qr,\overline{p}s+\overline{q}r)$ for any integers $\overline{p}$ and $\overline{q}$ satisfying $p\overline{q} - \overline{p}q = -1$. 
\end{theorem}

For example, the torus framing admissible transverse surgery on $T_{5,-8}$ in $(L(3,1),\xi_{5,-8})$ yields a connected sum of standard contact structures on $L(8,5)$ and $L(7,3)$. See Figure~\ref{fig:decoratedFarey2}.
\begin{figure}[htbp]{\scriptsize
  \vspace{0.2cm}
  \begin{overpic}%[grid,tics=10] 
  {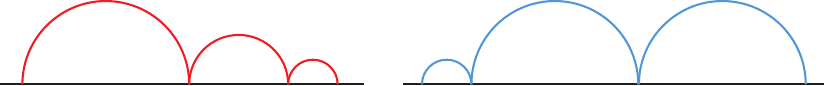}
    \put(47, 16){\Large \color{nred} $\circ$}
    \put(111, 8){\Large \color{nred} $+$}
    \put(147, 2){\large \color{nred} $\circ$}
    \put(344, 16){\Large \color{nblue} $\circ$}
    \put(261, 12){\Large \color{nblue} $-$}
    \put(211, 2){\large \color{nblue} $\circ$}
    \put(4, -10){$-3$}
    \put(84, -10){$-2$}
    \put(131, -10){$-\frac{5}{3}$}
    \put(154, -10){$-\frac{8}{5}$}

    \put(195, -10){$-\frac{8}{5}$}
    \put(219, -10){$-\frac{3}{2}$}
    \put(301, -10){$-1$}
    \put(386, -10){$0$}
  \end{overpic}}
  \vspace{0.3cm}
  \caption{Decorated paths in the Farey graph for the standard contact structures on $L(7,3)$ and $L(8,5)$.}
  \label{fig:decoratedFarey2}
\end{figure}

\begin{proof}
  The proof is essentially a reparametrization of Lemma~\ref{lem:pants-tori}. We first show that the torus framing admissible transverse surgery on $T_{p,q}$ is equivalent to Legendrian surgery on one of its Legendrian approximations. By Lemma~\ref{lem:binding}, there is a neighborhood $S_a$ of $T_{p,q}$ in $(L(r,s),\xi_{p,q})$ where $a$ is any number less than the page slope. Since $T_{p,q}$ is a non-trivial negative torus knot, Lemma~\ref{lem:fibered} says the torus framing plus one is less than the page slope so we can assume that $S_a$ has slope one larger than the torus framing. Thus as in the proof of Proposition~\ref{prop:negative=ot} we can assume that $S_a$ is a standard neighborhood of a Legendrian knot $L$ with contact framing one larger than the torus framing and $T_{p,q}$ is the transverse push-off of $L$. The proof of Proposition~\ref{prop:admissible-Legendrian} (\emph{c.f.} \cite[Lemma~3.16]{BaldwinEtnyre:transverse}) shows that torus framing admissible surgery on $T_{p,q}$ is the same as Legendrian surgery on $L$. 
  
  Since by Proposition~\ref{prop:negative=ot} the decorated path describing the contact structure $\xi_{p,q}$ is totally inconsistent at $q/p$ we know that inside of $L(r,s)$ we see the union of two basic slices
  \[
    B_\pm(p''/q'', p/q) \cup B_\mp(p/q,p'/q').
  \] 
   Now we change the coordinates using the following map  
  \[
    \phi = \begin{pmatrix}p' & -q' \\ -p & q\end{pmatrix}. %fixed
  \]
  Then $\phi$ sends $p/q \mapsto 0$, $p'/q' \mapsto \infty$ and $p''/q'' \mapsto -1$. As in the proof of Lemma~\ref{lem:pants-tori} we see there is a Legendrian knot $L$ realizing a $0$ sloped curve with contact framing one larger than the torus framing. Thus $\phi(L)$ is now a Legendrian knot in $B_\pm(-1,0) \cup B_\mp(0,\infty)$ such that the contact framing is one larger than the torus framing and \cite[Theorem~1.11]{ChakrabortyEtnyreMin20Pre} says that $\phi(L)$ is the Legendrian in Lemma~\ref{lem:pants-tori}. Now we apply Lemma~\ref{lem:pants-tori} and obtain a connected sum of two tight solid tori. Pulling back this contact manifold using $\phi^{-1}$, we obtain a connected sum of tight solid tori with meridian slope $p/q$. Thus the result of surgery is a connected sum of two lens spaces such that one has meridian slopes $p/q$ and $0$, and the other one has meridian slopes $-r/s$ and $p/q$. The first lens space is clearly $L(p,-q)$, and by changing the coordinates using the following map 
  \[
    \psi = \begin{pmatrix}\overline{p} & -\overline{q} \\ -p & q\end{pmatrix}%fixed
  \]
  sending $p/q \mapsto 0$ and $-r/s \mapsto (ps+qr)/(-\overline{p}s-\overline{q}r)$, we can see the second lens space should be $L(ps+qr,\overline{p}s+\overline{q}r)$. Thus we obtain
  \[
    L(p,-q) \,\#\, L(ps+qr,\overline{p}s+\overline{q}r). 
  \]

  Finally, since the signs of the edges in the decorated path clockwise of $p/q$ are the same, the contact structure on $L(p,-q)$ is universally tight by Theorem~\ref{thm:lens} so it is standard. Similarly, the contact structure on $L(ps+qr,\overline{p}s+\overline{q}r)$ is also standard. 
\end{proof}

%--------------------------------------------------------------------------------
\subsection{Symplectic pants cobordisms for Markov triples}\label{subsec:pants}
%--------------------------------------------------------------------------------
In this section we will build the symplectic cap for three lens spaces coming from a Markov triple. To identify these lens spaces we need a preliminary result about Markov triples. 

\begin{proposition}\label{prop:markovarith}
  For any Markov triple $(p_1,p_2,p_3)$, there exists a triple of positive integers $(q_1,q_2,q_3)$ satisfying the equations 
  \begin{align}
    p_3^2&=(p_1q_1-1)p_2^2+p_1^2(p_2q_2-1)\\
    p_3q_3-1&=p_2^2q_1^2+(p_1q_1+1)(p_2q_2-1)\\
    q_i &= \pm 3p_j/p_k \pmod{p_i}\\
    q_1 &< 0 
  \end{align}
  where $(i, j, k)$ is a permutation of $1$, $2$, $3$.
\end{proposition}

\begin{proof}
  Pick two integers $x > 0$ and $y < 0$ satisfying $p_1x+p_2y = 1$. Then define
  \begin{align*}
    q_1 &\coloneqq 3p_3y \\
    q_2 &\coloneqq 3p_3x \\
    q_3 &\coloneqq -3p_1y+3p_2x + 9p_2p_3y
  \end{align*}

  The fact that the choices of $q_1$ and $q_2$ satisfy the third condition is immediate. We will show that they satisfy the first condition. 
  \begin{align*}
    (p_1q_1-1)p_2^2+p_1^2(p_2q_2-1) &= 3p_1p_2^2p_3y + 3p_1^2p_2p_3x - p_1^2 - p_2^2\\
    &=3p_1p_2p_3 - p_1^2 - p_2^2\\
    &= p_3^2
  \end{align*}

  The definition of $q_3$ is chosen so that the second condition is satisfied, as can easily be checked by noting that $p_1q_2+p_2q_1=3p_3$. To show that it satisfies the third condition, it's enough to prove that $p_1y - p_2x =\pm p_2/p_1 \pmod{p_3}$. To see that, observe
  \begin{align*}
  p_1(p_1y - p_2x) &\equiv -p_2^2y - p_1p_2x \pmod{p_3} &\text{ since } p_1^2 \equiv -p_2^2 \pmod{p_3}\\
  &\equiv -p_2 \pmod{p_3}
  \end{align*}
  
 Note that this shows $q_3 \neq 0$. The fourth condition is immediate from $y < 0$.
\end{proof}

\begin{definition}\label{def:pants}
  Let $(p_1,p_2,p_3)$ be a Markov triple and $(q_1,q_2,q_3)$ be a triple of integers from Proposition~\ref{prop:markovarith}. We call a compact symplectic $4$-manifold $(P,\omega_P)$ \dfn{a (concave) symplectic pants cobordism for $(p_1,p_2,p_3)$} if $b_2(P)=1$ and $(P,\omega_P)$ is a strong symplectic cap with three concave boundary components 
  \[
    \bigsqcup_{i=1}^3 L(p_i^2,p_iq_i-1).
  \]
\end{definition}

\begin{remark}
  Notice that the orientations of concave boundary components are the opposite of the ordinary boundary orientations. Thus a symplectic pants cobordism is a smooth cobordism from $L(p_3^2,p_3q_3-1)$ to $L(-p_1^2,p_1q_1-1) \sqcup L(-p_2^2,p_2q_2-1)$.
\end{remark}

\begin{theorem}\label{thm:symplectic-pants}
  For any Markov triple $(p_1,p_2,p_3)$, there exists a symplectic pants cobordism $(P,\omega_P)$ for $(p_1,p_2,p_3)$ such that the induced contact structure on each boundary component is standard.  Further, $(P,\omega_P)$ admits a (relative) symplectic handlebody decomposition consisting of one convex-concave $2$-handle and one concave $3$-handle attached to $L(p_3^2,p_3q_3-1)$. 
\end{theorem}

\begin{proof}
  Let $K$ be a $(-p_1q_1-1, p_1^2)$-torus knot in $(L(p_3^2,p_3q_3-1), \xi_{std})$. We first claim that $K$ is a positive torus knot in $L(p_3^2,p_3q_3-1)$. First, by Proposition~\ref{prop:markovarith} we know $q_1 \leq 0$. Assume $q_1=0$. Then according to the proof of Proposition~\ref{prop:markovarith}, we have $x=1$ and $y=0$, which implies that $p_1 = 1$ and $q_3=3p_2$. Thus the defining equation $p_1^2+p_2^2+p_3^2=3p_1p_2p_3$ gives us $p_3^2=3p_2p_3-1-p_2^2$ and so we have
  \[
    \frac{p_1^2}{-p_1q_1-1} = -1 < -\frac{p_3^2}{p_3q_3-1} = -\frac{p_3^2}{3p_2p_3-1} = -\frac{p_3^2}{p_2^2+p_3^2}< 0
  \]
  and $K$ is a positive torus knot by the definition in Section~\ref{subsec:torus}. Now assume $q_1 < 0$. There are two cases we need to consider. First, suppose $q_3 > 0$. Then we have 
  \[
    -\frac{p_3^2}{p_3q_3-1} < 0 < \frac{p_1^2}{-p_1q_1-1}
  \]
  so $K$ is a positive torus knot. Next, suppose $q_3 < 0$. Then one can see that $0 > -p_2^2$ implies that $p_1^2(p_2^2q_1^2 + (p_1q_1+1)(p_2q_2-1)) > ((p_1q_1-1)p_2^2 +p_1^2(p_2q_2-1))(p_1q_1+1)$ and using Proposition~\ref{prop:markovarith} that inequality implies that $p_1^2(p_3q_3-1) > p_3^2(p_1q_1+1)$, which of course gives 
  \[
    \frac{p_1^2}{-p_1q_1-1} < -\frac{p_3^2}{p_3q_3-1}
  \]
  Combining with the fact $p_1^2/(-p_1q_1-1) > 0$, we can conclude that $K$ is a positive torus knot.
  
  Therefore, the torus framing of $K$ is greater than the page slope by Lemma~\ref{lem:fibered}. Also by Theorem~\ref{prop:postive=univ}, $K$ supports a standard contact structure. Thus $K$ satisfies the hypothesis of Theorem~\ref{thm:cap}, so we obtain a strong symplectic cap $(C,\omega_C)$ by attaching a convex-concave $2$-handle along $K$ with the torus framing. Recall that we orient the concave boundary of a symplectic cap as the opposite of the ordinary boundary orientation, so the resulting contact $3$-manifold is the result of the torus framing admissible transverse surgery on a negative torus knot $\overline{K}$, the $(p_1q_1+1, p_1^2)$-torus knot supporting the contact structure $\xi_{p_1q_1+1, p_1^2}$ on $L(-p_3^2,p_3q_3-1)$. It is straightforward to check
  \[
    p_1^2 q_1^2 - (p_1q_1-1)(p_1q_1+1) = 1.
  \]
  Therefore by Theorem~\ref{thm:nonloose-surgery}, the result of the surgery is a connected sum of standard contact structures on  
  \[
      L(p_1^2, -p_1q_1-1) \,\#\, L(p_1^2(p_3q_3-1)+(p_1q_1+1)(-p_3^2), (p_1q_1-1)(p_3q_3-1) - q_1^2 (-p_3^2)).
  \]
  Since $(-p_1q_1-1)(p_1q_1-1) \equiv 1\pmod{p_1^2}$, the first lens space is diffeomorphic to $L(p_1^2,p_1q_1-1)$. Also by Proposition~\ref{prop:markovarith}, we have
  \[
    \begin{pmatrix}
      -p_1 q_1-1 & q_1^2 \\
      -p_1^2 & p_1q_1-1\\
    \end{pmatrix}
    \begin{pmatrix} 
      1-p_2q_2\\
      p_2^2\\
    \end{pmatrix}
    =
    \begin{pmatrix}
      p_2^2 q_1^2+(p_1 q_1+1)(p_2 q_2-1)\\
      p_2^2(p_1 q_1-1)+p_1^2(p_2 q_2-1)\\
    \end{pmatrix}
    =
    \begin{pmatrix}
      p_3 q_3-1\\
      p_3^2\\
    \end{pmatrix},
  \]
  and this implies
  \[
    \begin{pmatrix}
      p_1q_1-1 & -q_1^2 \\
      p_1^2 & -p_1 q_1-1  \\
    \end{pmatrix}
    \begin{pmatrix}
      p_3 q_3-1\\
      p_3^2\\
    \end{pmatrix}
    =
    \begin{pmatrix}
      -q_1^2p_3^2 + (p_1q_1-1)(p_3q_3-1)\\
      -(p_1q_1 + 1)p_3^2 - p_1^2(1-p_3q_3)\\
    \end{pmatrix}
    =
    \begin{pmatrix} 
      1-p_2q_2\\
      p_2^2\\
    \end{pmatrix},
  \]
  Thus the second lens space is $L(p_2^2, p_2q_2-1)$.
  
  Separately, observe that we can attach a Weinstein $1$-handle to (a part of the symplectizations of)
  \[
    (L(p_1^2,p_1q_1-1), \xi_{std}) \;\;\text{and}\;\; (L(p_2^2,p_2q_2-1), \xi_{std})
  \] 
  and obtain a Weinstein cobordism $(W,\omega_W)$ from 
  \[
    (L(p_1^2,p_1q_1-1), \xi_{std}) \sqcup (L(p_2^2,p_2q_2-1), \xi_{std})
  \] 
  to 
  \[
    \left(L(p_1^2,p_1q_1-1) \,\#\, L(p_2^2,p_2q_2-1), \xi_{std} \,\#\, \xi_{std}\right).
  \] 

  Now we have a symplectic cap $(C,\omega_C)$ and a Weinstein cobordism $(W,\omega_W)$. The concave boundary of $(C,\omega_C)$ and the convex boundary of $(W,\omega_W)$ are contactomorphic, so we can glue $(C,\omega_C)$ and $(W,\omega_W)$ together along $L(p_2^2,p_2q_2-1) \,\#\, L(p_3^2,p_3q_3-1)$ and obtain the desired pants cobordism.  Since a Weinstein $1$-handle can be considered a concave $3$-handle when turned upside down, we can obtain the pants cobordism by attaching a convex-concave $2$-handle and a concave $3$-handle to $L(p_3^2,p_3q_3-1)$. See Figure~\ref{fig:pants} for a schematic picture for the pants cobordism.
\end{proof}

\begin{figure}[htbp]{
  \vspace{0.2cm}
  \begin{overpic}[tics=20]{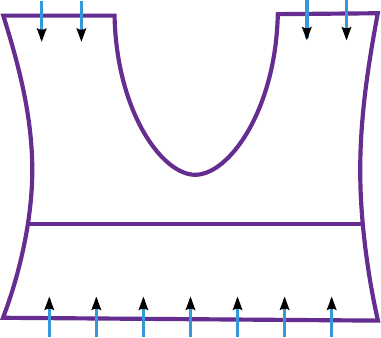}
   \put(55, 65){concave $3$-handle}
   \put(35, 30){convex-concave $2$-handle}
  \end{overpic}}
  \vspace{0.2cm}
  \caption{A schematic picture for a concave pants cobordism.}
  \label{fig:pants}
\end{figure}

Let $(p_1,p_2,p_3)$ be a Markov triple and $(P,\omega_P)$ is a (concave) symplectic pants cobordism for $(p_1,p_2,p_3)$ from Theorem~\ref{thm:symplectic-pants}. Since the induced contact structure on each concave boundary component of $(P,\omega)$ is standard, we can glue three Weinstein $B_{p_i,q_i}$ along the boundary components and obtain a closed symplectic $4$-manifold $(X_{p_1,p_2,p_3},\omega_{p_1,p_2,p_3})$. This construction is unique up to diffeomorphism since any contactomorphism of a standard contact structure on $L(p_i^2,p_iq_i-1)$ extends to a diffeomorphism of $B_{p_i,q_i}$ according to Remark~\ref{rem:pushingpast}. From the handlebody viewpoint, we consider this as starting from Weinstein $B_{p_3,q_3}$, attaching a pants cobordism $(P,\omega_P)$ (equivalently attaching a convex-concave $2$ handle and a concave $3$-handle, see Theorem~\ref{thm:symplectic-pants}). After that, we attach the upside down Weinstein $B_{p_1,q_1}$ and $B_{p_2,q_2}$ to each concave boundary component of the pants cobordism. Since we completely understand the Weinstein handlebody structure of each $B_{p_i,q_i}$, this gives a symplectic handlebody decomposition of $(X_{p_1,p_2,p_3},\omega_{p_1,p_2,p_3})$.

The next proposition summarizes the discussion above.

\begin{proposition}\label{prop:pants-to-closed}
  For any Markov triple $(p_1,p_2,p_3)$, there exists a closed symplectic $4$-manifold $(X_{p_1,p_2,p_3},\omega_{p_1,p_2,p_3})$ built from the symplectic pants cobordism for $(p_1,p_2,p_3)$ and $B_{p_i,q_i}$. Moreover, $(X_{p_1,p_2,p_3},\omega_{p_1,p_2,p_3})$ admits a symplectic handlebody decomposition consisting of a convex $0$-handle, a convex $1$-handle, a convex $2$-handle, a convex-concave $2$-handle, two concave $2$-handles, three concave $3$-handles, and two concave $4$-handles.  
\end{proposition}

Currently, it is not clear that  $X_{p_1, p_2, p_3}$ is $\cp$ or that the symplectic structure $\omega_{p_1,p_2,p_3}$ is deformation equivalent to a symplectic structure on $\cp$. In the following two sections, we will show that $X_{p_1, p_2, p_3}$ is indeed $\cp$ and that $\omega_{p_1,p_2,p_3}$ is deformation equivalent to the standard symplectic structure on $\cp$.

%%%%%%%%%%%%%%%%%%%%%%%%%%%%%%%%%%%%%%%%%%%%%%%%%%%%%%%%%%%%%%%%%%%%%%%%%%%%%%%%%%%%
\section{From symplectic handlebodies to horizontal decompositions}\label{sec:top}
%%%%%%%%%%%%%%%%%%%%%%%%%%%%%%%%%%%%%%%%%%%%%%%%%%%%%%%%%%%%%%%%%%%%%%%%%%%%%%%%%%%%

In order to exhibit explicitly the symplectic structure on the manifolds $X_{p_1,p_2,p_3}$ in Section~\ref{sec:geometry}, it is convenient to use the pants construction outlined in Section~\ref{subsec:pants}. However, in order to identify these manifolds as $\cp$, it is convenient to use a construction known as \emph{horizontal handle decompositions}, introduced in recent work of Lisca and Parma, \cite{LiscaParma:horiz}. In Section~\ref{subsec:horiz}, we will introduce horizontal handle decompositions and recall the relevant results from \cite{LiscaParma:horiz,LiscaParma:horiz2}. In Section~\ref{subsec:convert}, we will convert our construction from Section~\ref{subsec:pants} into a horizontal handle diagram as shown in Figure~\ref{fig:handlebody}. This conversion will be used to prove Theorem~\ref{thm:sympCP2} that identifies $X_{p_1,p_2,p_3}$ with $\cp$ in Section~\ref{subsec:sympemb}, which completes our proof of Theorem~\ref{thm:main1} and explicitly symplectically embeds three rational homology balls associated to Markov triples into $\cp$. 

%-----------------------------------------------------------------
\subsection{Horizontal handle decompositions}\label{subsec:horiz}
%-----------------------------------------------------------------
Horizontal handle decompositions were developed by Lisca and Parma in  \cite{LiscaParma:horiz}. The following treatment will be brisk; for more details, see \cite{LiscaParma:horiz}. 

Let $H$ be a handle decomposition of a closed 4-manifold $X$. By a standard argument, $H$ can be assumed to have a unique 0-handle and a unique 4-handle. Let $H_1$ denote the sub-handlebody consisting of the 0- and 1-handles. 

\begin{definition}
  A handle decomposition $H$ is \emph{horizontal} with genus $g$ if $\bd(H_1)$ admits a genus $g$ Heegaard splitting surface $\Sigma$ such that
  \begin{enumerate}
    \item There is some order $\{h_i:i\in\mathbb{N}\}$ on the set of 2-handles of $H$ such that the 2-handles can be isotoped in $\bd H_1$ so that in some neighborhood $\Sigma\times [0,1]$ of $\Sigma$, the attaching sphere of $h_i$ is a non-separating simple closed curve on $\Sigma\times\{1-1/i\}$.
    \item each 2-handle framing $f_i$ is equal to $f_\Sigma\pm 1$ where $f_\Sigma$ is the framing induced by $\Sigma$.
  \end{enumerate}
\end{definition}

The requirements about $\Sigma$ and the 2-handle framings in the definition of a horizontal handle decomposition cause the 2-handle attachments to naturally modify a Heegaard splitting of the boundary.  More precisely,  let $H$ be a horizontal handle decomposition and let $H_1^j$ denote the sub-handlebody of $H$ consisting of $H_1$ together with $h_1, \ldots, h_j$, the first $j$ 2-handles. Lemma \ref{lem:surfacetwist} guarantees that for all $j$, $\bd H_1^j$ inherits a natural genus $g$ Heegaard splitting.  The theory of horizontal handle diagrams is thus particularly helpful to demonstrate disjoint embeddings of a collection of 3-manifolds with bounded Heegaard genus into a 4-manifold. 

Lisca and Parma utilize this theory to give smooth embeddings of collections of lens spaces into $\cp$. They also give classifications of the smooth 4-manifolds realized by horizontal decompositions with small genus. We will use one of their classification results. For the statement, recall that an essential simple closed curve on $T^2$ can be written as $\gamma = q\lambda_U + p\mu_U$ where $\mu_U$ and $\lambda_U$ are curves on $T^2$ that form a symplectic basis of $H_1(T^2)$. 

\begin{theorem}[Lisca--Parma \cite{LiscaParma:horiz2}]\label{thm:LiscaParma}
  Let $X$ be a closed oriented 4-manifold with a horizontal decomposition of genus one having one 0-handle, one 1-handle, three 2-handles, one 3-handle, and one 4-handle. Suppose that the 2-handles are attached along essential simple closed curves $\gamma_1$, $\gamma_2$, and $\gamma_3$ such that each has framing $-1$ relative to the surface framing. Let $x_1:=\gamma_2\cdot\gamma_3$, $x_2:=\gamma_1\cdot\gamma_3$, and $x_3\colon = \gamma_1\cdot\gamma_2$. If $(x_1,x_2,x_3)\neq (0,0,0)$ and $x_1^2+x_2^2+x_3^2=x_1x_2x_3$, then $X$ is diffeomorphic to $\cp$. 
\end{theorem}

To prove Theorem~\ref{thm:sympCP2} that identifies $X_{p_1,p_2,p_3}$ with $\cp$, we will show in Section~\ref{subsec:convert} that the symplectic 4-manifolds $(X_{p_1,p_2,p_3},\omega_{p_1,p_2,p_3})$ we constructed in Section~\ref{subsec:pants} in fact admit a horizontal handle decompositions which have the form in Theorem~\ref{thm:LiscaParma}. The conclusion that our symplectic manifolds are in fact diffeomorphic to $\cp$ then follows from Theorem \ref{thm:LiscaParma}. 

%------------------------------------------------------------------------
\subsection{Converting pants to horizontal handles}\label{subsec:convert}
%------------------------------------------------------------------------
In order to recognize the closed symplectic 4-manifolds we built in Proposition \ref{prop:pants-to-closed}, we would like to draw an explicit horizontal handle diagrams of these manifolds, and then identify it using Theorem~\ref{thm:LiscaParma}. When we built the symplectic pants cobordism in Theorem \ref{thm:symplectic-pants}, the pants cobordism is described by attaching a 2-handle to a lens space; the resulting top boundary is a connected sum of lens spaces. We will find it convenient here to take the dual perspective.  We begin in Proposition~\ref{prop:feet-to-waist} by describing a 2-handle attachment to $B_{p_1,q_1}\,\natural\, B_{p_2,q_2}$ which results in a 4-manifold $C$ with $\partial C=L(-p_3^2,p_3 q_3-1)$, where $\natural$ denotes the boundary connected sum. Then in Proposition \ref{prop:pants-equals-pants} we show that this 2-handle cobordism from $L(p_1^2,p_1q_1-1)\,\#\,L(p_2^2,p_2q_2-1)$ to $L(-p_3^2,p_3 q_3-1)$ is indeed the pants cobordism from Theorem~\ref{thm:symplectic-pants} turned upside down. Finally, in Proposition~\ref{prop:final-ball} we give a handle diagram for the result of attaching $B_{p_3,q_3}$ to $C$; this final handle diagram describes the closed manifold $X_{p_1,p_2,p_3}$ of Proposition~\ref{prop:pants-to-closed}.  We will then prove Theorem~\ref{thm:sympCP2} by applying Theorem~\ref{thm:LiscaParma} to our final horizontal handle diagram. 

\begin{proposition}\label{prop:feet-to-waist}
  For pairs of integers $(p_i, q_i)$ satisfying the conditions in Proposition \ref{prop:markovarith}, when a 2-handle is attached to the boundary connect sum $B_{p_1,q_1}\,\natural\, B_{p_2,q_2}$ with the attaching sphere and framing demonstrated in yellow in Figure~\ref{fig:pantsconvert1}, we obtain the 4-manifold given in Figure~\ref{fig:pantsconvert}, which has boundary $L(-p_3^2,p_3q_3-1)$. 
\end{proposition}

\begin{proof}
  Figure~\ref{fig:pantsconvert1}, \ref{fig:pantsconvert2}, and ~\ref{fig:pantsconvert} gives the 4-manifold diffeomorphism claimed. 

  \begin{figure}[htbp]{\small
    \vspace{0.2cm}
    \begin{overpic}%[grid,tics=10] 
    {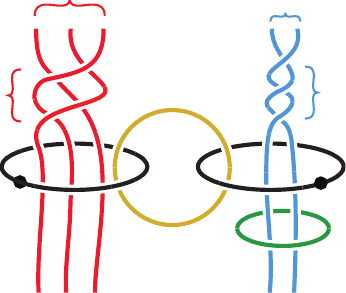}
      \put(-14, 94){\color{nred}$-q_1$}
      \put(30, 146){\color{nred}$p_1$}     
      \put(130, 139){\color{lblue}$-p_2$}
      \put(155, 94){\color{lblue}$q_2$}     
      \put(100,29){\color{l2yellow}$0$}
      
      \put(155, 18){\color{dgreen}$\langle 0\rangle$}     
    \end{overpic}}
    \vspace{0.3cm}
    \caption{The boundary sum $B_{p_1,q_1}\,\natural\, B_{p_2,q_2}$ with a $2$-handle attached.  All the vertical 2-handles are $-1$ framed with respect to the torus framing. The lower-right (green) unknot is not part of the surgery diagram, but a framed curve we will be watching. It is helpful to note that the green curve is the meridian to the yellow curve, as can be seen by sliding the green curve over the left-hand $1$-handle (thought of as a $0$-framed $2$-handle). }
    \label{fig:pantsconvert1}
  \end{figure}

  From Figure~\ref{fig:pantsconvert1} to the left hand side of Figure~\ref{fig:pantsconvert2}, slide the blue $2$-handle over the yellow $2$-handle several times until it is disjoint from the left most $1$-handle and then cancel the right most $1$-handle with the yellow $2$-handle. Then in Figure~\ref{fig:pantsconvert2} we convert the dotted circle notation for a $1$-handle to the standard representation by explicitly drawing the attaching sphere of the $1$-handle, recall this is two $3$-balls, in the figure one is the balls is contained in the innermost $2$-sphere while the second is outside of the outermost $2$-sphere. We note that we can explicitly see the boundary of the $0$ and $1$-handle, which is $S^1\times S^2$ by identifying the innermost and outermost $2$-spheres in the figure. 
  %we cancel one pair of $1$- and $2$-handles, and convert from a standard handle diagram to explicitly drawing the attaching spheres of the 2-handles in $S^1\times S^2$.  
  From the left hand side to the right hand side of Figure~\ref{fig:pantsconvert2}, we isotope the attaching spheres in $S^1\times S^2$. From the right frame of Figure~\ref{fig:pantsconvert2} to Figure~\ref{fig:pantsconvert}, we convert back into dotted circle notation. 

  \begin{figure}[htbp]{\small
    \vspace{0.2cm}
    \begin{overpic}%[grid,tics=10] 
    {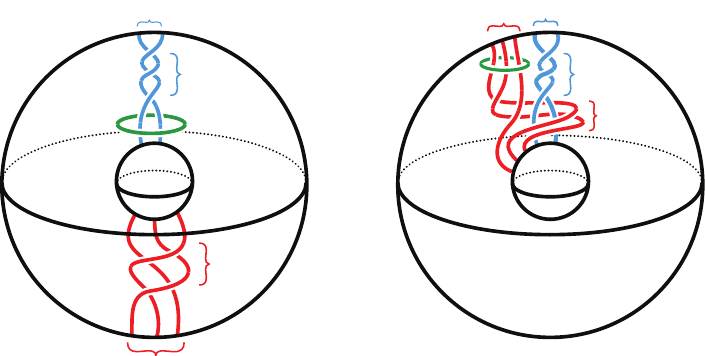}
      \put(72, -8){\color{nred}$p_1$}
      \put(103, 42){\color{nred}$-q_1$}
      \put(66, 167){\color{lblue}$-p_2$}
      \put(90, 134){\color{lblue}$q_2$}   
      \put(48, 118){\color{dgreen}$\langle 0\rangle$}  
        
      \put(240, 164){\color{nred}$p_1$}
      \put(290, 114){\color{nred}$q_1$}
      \put(256, 167){\color{lblue}$-p_2$}
      \put(279, 134){\color{lblue}$q_2$}   
      \put(219, 131){\color{dgreen}$\langle 0\rangle$}   
    \end{overpic}}
    \vspace{0.3cm}
    \caption{Attaching spheres of $2$-handles in $S^1\times S^2$. All red and blue 2-handles are $-1$ framed with respect to the torus framing. Angled brackets denote a framing we are watching, not an attaching instruction.}
    \label{fig:pantsconvert2}
  \end{figure}

  \begin{figure}[htbp]{\small
    \vspace{0.2cm}
    \begin{overpic}%[grid,tics=10] 
    {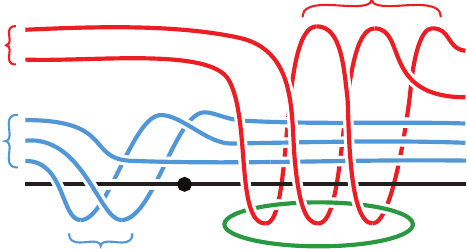}
      \put(179, 124){\color{nred}$p_1$}
      \put(-9, 96){\color{nred}$q_1$}
      \put(41, -6){\color{lblue}$-p_2$}
      \put(-9, 51){\color{lblue}$q_2$}   
      \put(195, -3){\color{dgreen}$\langle 0\rangle$}  
    \end{overpic}}
    \vspace{0.3cm}
    \caption{All red and blue 2-handles are $-1$ framed with respect to the torus framing. Angled brackets denote a framing we are watching, not an attaching instruction.}
    \label{fig:pantsconvert}
  \end{figure}

  That the boundary of Figure~\ref{fig:pantsconvert} is a lens space follows from Lemma~\ref{lem:surfacetwist}.  To compute which lens space, we use the method we set up in Remark~\ref{rem:pushingpast}. In the proof of Lemma~\ref{lem:Bpq} we already computed that the result of pushing the longitude,  $\begin{pmatrix} 1 \\ 0 \end{pmatrix}$, of black curve past the blue surgery curve is $\begin{pmatrix} 1-p_2q_2 \\ p_2^2 \end{pmatrix}$. As in Remark \ref{rem:pushingpast}, we compute that subsequently pushing past red surgery curve gives
  \[
    \begin{pmatrix}
      1+p_1 q_1  & -q_1^2\\ 
      p_1^2 & 1-p_1q_1 
    \end{pmatrix}
    \begin{pmatrix} 
      1-p_2q_2\\
      p_2^2\\
    \end{pmatrix}
    =
    \begin{pmatrix}
      -p_2^2q_1^2-(p_1q_1+1)(p_2q_2-1)\\
      p_1^2(1-p_2q_2)+p_2^2(1-p_1q_1)\\
    \end{pmatrix}
    =
    \begin{pmatrix}
      -(p_3 q_3-1)\\
      -p_3^2\\
    \end{pmatrix},
  \]
  where the last equality comes from Proposition \ref{prop:markovarith}. 
  \end{proof}

Let $W_{p_1,p_2,q_1,q_2}$ denotes this 1- and (yellow) 2-handle cobordism from $L(p_1,q_1)\,\sqcup\, L(p_2,q_2)$ to $L(-p_3^2,p_3q_3-1)$ shown in Figure~\ref{fig:pantsconvert1}. 

\begin{proposition}\label{prop:pants-equals-pants}
  If the pairs of integers $(p_i,q_i)$ satisfy the equations in Proposition \ref{prop:markovarith}, then $W_{p_1,p_2,q_1,q_2}$ is diffeomorphic to the pants cobordism from Theorem \ref{thm:symplectic-pants}. 
\end{proposition}

\begin{proof}
Recall from Section~\ref{subsec:handles}, that a $2$-handle is $D^2\times D^2$ attached to a manifold along $(\partial D^2)\times D^2$ and we call $\partial D^2\times\{0\}$ the attaching sphere and $\{0\}\times \partial D^2$ the belt sphere, so if we turn the handlebody picture upside down, then the belt sphere becomes the attaching sphere. See \cite[Sections~4.1 and~5.5]{GompfStipsicz99}. We also recall that $D^2\times \{0\}$ is called the core and $\{0\}\times D^2$ is called the co-core of the $2$-handle.
Thus to prove the proposition, it suffices to locate the belt sphere of the 2-handle of $W_{p_1,p_2,q_1,q_2}$ in $L(-p_3^2,p_3q_3-1)$ and show that this agrees with the framed 2-handle described in Theorem~\ref{thm:symplectic-pants}. 

  We have exhibited the belt sphere of the yellow 2-handle of $W_{p_1,p_2,q_1,q_2}$ in Figure~\ref{fig:pantsconvert1} in green. The cocore disk 0-frames the belt sphere. To conclude we must identify this green curve in $L(-p_3^2,p_3q_3-1)$. To do this, observe that in Figure~\ref{fig:pantsconvert} we can think of this belt sphere as living in an embedded torus in $L(-p_3^2,p_3q_3-1)$ which is located between the blue and the red surgery curves; in this torus the surface framing agrees with the 0-framing. Then to identify this framed belt sphere in the (outermost) Heegaard torus for $L(-p_3^2,p_3q_3-1)$, we simply must push it past the red surgery curve, and the new framing will be the new surface framing. As in the proof of Proposition \ref{prop:feet-to-waist}, we compute 
  \[  
    \begin{pmatrix}
    1+p_1 q_1  & -q_1^2\\ 
    p_1^2 & 1-p_1q_1 
    \end{pmatrix}
    \begin{pmatrix} 
    1\\
    0\\
    \end{pmatrix}
    =
    \begin{pmatrix}
    p_1 q_1+1\\
    p_1^2\\
    \end{pmatrix}
  \]
which implies that the belt sphere is the $(p_1q_1 + 1, p_1^2)$-torus knot in $L(-p_3^2,p_3q_3-1)$ and this completes the proof.
\end{proof}

\begin{figure}[htbp]{\small
  \vspace{0.2cm}
  \begin{overpic}[tics=20]{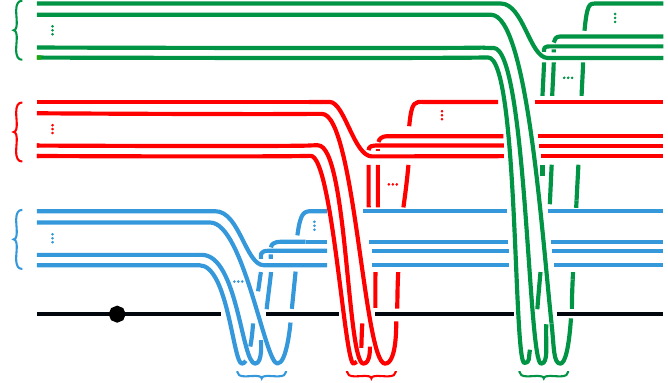}
  \put(-5,168){\color{dgreen}$q_3$}
  \put(-5,120){\color{red}$q_1$}
  \put(-5,68){\color{lblue}$q_2$}
  \put(118,-6){\color{lblue}$-p_2$}
  \put(176,-6){\color{red}$p_1$}
  \put(260,-6){\color{dgreen}$p_3$}
  \put(300,10){\normalsize$\cup$ $3$-handle, $4$-handle}
  \end{overpic}}
  \vspace{0.3cm}
  \caption{Horizontal handle decomposition of $X_{p_1,p_2,p_3}$. All three 2-handles are $-1$ framed with respect to the torus framing.}
  \label{fig:handlebody}
\end{figure}

\begin{proposition} \label{prop:final-ball}
  If the pairs of integers $(p_i,q_i)$ satisfy the equations in Proposition \ref{prop:markovarith}, then the 4-manifold presented in Figure \ref{fig:handlebody} is orientation preserving diffeomorphic to $X_{p_1,p_2,p_3}$.
\end{proposition}

\begin{proof}
  We have demonstrated already that there is an orientation preserving diffeomorphism from the handle diagram $H$ in Figure~\ref{fig:pantsconvert} to the codimension 0 submanifold of $X_{p_1,p_2,p_3}$ obtained by removing the bottom rational ball, $B(p_3,q_3)$.  So it remains to demonstrate that adding the additional handles in Figure \ref{fig:handlebody} to $H$ (i.e. the green 2-handle, and the 3- and 4-handle) corresponds exactly to regluing the rational ball $B(p_3,q_3)$.

  Observe (say in Figure \ref{fig:Bpq}) that the 0-framed cocore of the 2-handle of $B(p_3,q_3)$ is the 1-framed $-p_3/q_3$ curve on the Heegaard torus for $L(p_3^2,p_3q_3-1)$ and then fill with $B^3 \times S^1$ given by the union of $3-$ and $4-$handles.  To get a handle diagram of $B(p_3,q_3)$ upside down, we must attach a $(-1)$-framed 2-handle to $L(-p_3^2,p_3q_3-1)\times I$ along the $p_3/q_3$ curve on the Heegaard torus.  Since this is exactly the way the green 2-handle is attached to $H$, we see that the 4-manifold presented in Figure \ref{fig:handlebody} is orientation preserving diffeomorphic to $X_{p_1,p_2,p_3}$.
\end{proof}

%------------------------------------------------------------------------------------------------------------------
\subsection{Symplectic embeddings of rational homology balls into \texorpdfstring{$\cp$}{CP2}}\label{subsec:sympemb}
%------------------------------------------------------------------------------------------------------------------
In Section~\ref{subsec:pants}, we constructed a closed symplectic $4$-manifold $(X_{p_1,p_2,p_3}, \omega_{p_1,p_2,p_3})$ for each Markov triple $(p_1,p_2,p_3)$ and in Section~\ref{subsec:convert}, we showed that $X_{p_1,p_2,p_3}$ admits a genus one horizontal handlebody decomposition. Let $\omega_{std}$ be the standard symplectic structure on $\cp$. We now show that it is immediate from Theorem~\ref{thm:LiscaParma} and the result of Taubes~\cite{Taubes:SWGr} that our manifold is the standard symplectic $\cp$ (after scaling the symplectic form). 

\begin{theorem}\label{thm:sympCP2}
  The  manifold $(X_{p_1,p_2,p_3},\omega_{p_1,p_2,p_3})$ is deformation equivalent to $(\cp,\omega_{std})$.
\end{theorem}

\begin{proof}
  By Proposition~\ref{prop:final-ball}, we know $X_{p_1,p_2,p_3}$ is diffeomorphic to the closed $4$-manifold shown in Figure~\ref{fig:handlebody}. With the notations in Theorem~\ref{thm:LiscaParma}, we have 
  \begin{align*}
    \gamma_1 = -p_2\mu_U + q_2\lambda_U, \quad
    \gamma_2 &= p_1\mu_U + q_1\lambda_U, \quad
    \gamma_3 = p_3\mu_U + q_3\lambda_U.
  \end{align*}
  We also have
  \begin{align*}
    x_1 &= \gamma_2 \cdot \gamma_3 = p_1q_3 - p_3q_1,\\
    x_2 &= \gamma_1 \cdot \gamma_3 = -p_2q_3 - p_3q_2,\\
    x_3 &= \gamma_1 \cdot \gamma_2 = -p_2q_1-p_1q_2.
  \end{align*}
  Using Proposition~\ref{prop:markovarith} one can show (through a non-trivial computation, carried out with the help of Mathematica) that the $x_i$ satisfy the hypothesis in Theorem~\ref{thm:LiscaParma}, and so it is immediate that $X_{p_1,p_2,p_3}$ is diffeomorphic to $\cp$. According to the result of Taubes~\cite[Theorem~0.3]{Taubes:SWGr}, there exists a unique symplectic structure on $\cp$ up to symplectic deformation. Thus $(X_{p_1,p_2,p_3},\omega_{p_1,p_2,p_3})$ is symplectomorphic to $(\cp,\omega_{std})$ after scaling $\omega_{p_1,p_2,p_3}$. 
\end{proof}

\begin{remark}
  In Section~\ref{sec:toric}, we will exhibit two more proofs of Theorem~\ref{thm:sympCP2} using almost toric geometry. 
\end{remark}

\begin{proof}[Proof of Theorem~\ref{thm:main1}]
  By Proposition~\ref{prop:pants-to-closed}, we know that $(X_{p_1,p_2,p_3},\omega_{p_1,p_2,p_3})$ was built from $B_{p_3,q_3}$, which consists of convex (Weinstein) $0$-, $1$-, $2$-handles, by attaching a convex-concave $2$-handle, followed by gluing $B_{p_1,q_1}\,\natural\,B_{p_2,q_2}$ together along boundary. Here we consider $B_{p_1,q_1}\,\natural\,B_{p_2,q_2}$ as an upside down Weinstein domain. Thus $(X_{p_1,p_2,p_3},\omega_{p_1,p_2,p_3})$ consists of a convex $0$-handle, a convex $1$-handle, a convex $2$-handle, a convex-concave $2$-handle, two concave $2$-handles, three concave $3$-handles, and two concave $4$-handles. Combining with Theorem~\ref{thm:sympCP2}, this provides a symplectic handlebody decomposition of $(\cp,\omega_{std})$ in which we explicitly see embeddings of the rational homology balls $B_{p_i,q_i}$.
\end{proof}

%%%%%%%%%%%%%%%%%%%%%%%%%%%%%%%%%%%%%%%%%%%%%%%%%%%%%%%%%%%%%%%%%%%%%%%%%%%%%%%%%%%%%%%%%%%%%%%%
\section{Mutation and the almost toric geometry of \texorpdfstring{$\cp$}{CP2}}\label{sec:toric}
%%%%%%%%%%%%%%%%%%%%%%%%%%%%%%%%%%%%%%%%%%%%%%%%%%%%%%%%%%%%%%%%%%%%%%%%%%%%%%%%%%%%%%%%%%%%%%%%

In this section, we will give two additional, self-contained, proofs that $X_{p_1,p_2,p_3}$ is diffeomorphic to $\C P^2$, and hence two more proofs of Theorem~\ref{thm:main1}. The first proof is based on almost toric geometry, though does not actually use it. We begin in Section~\ref{subsec:mutate} by giving this alternate proof. 

In Section~\ref{subsec:toric}, we give an overview of almost toric pictures and discuss Vianna's \cite{Vianna:exotic, Vianna:tori} construction of infinitely many almost toric pictures of $\cp$ corresponding to Markov triples. We will give explicit handle descriptions corresponding to these almost toric pictures of $\cp$ and transferring the cut (giving yet another proof that $X_{p_1,p_2,p_3}$ is diffeomorphic to $\C P^2$) in Section~\ref{sec:atf2hbd}. In the literature the relationship between almost toric pictures and symplectic handlebodies has been studied for Weinstein domains \cite{ACGMMSW:Weinstein}, but not for closed symplectic manifolds. This handlebody description of transferring the cut is what lies behind our proof of Theorem~\ref{thm:main1} in Section~\ref{subsec:mutate}.  

Finally in Section~\ref{sec:hbd2atf}, we will see that the symplectic structure coming from our symplectic handlebody decomposition can be deformed so that one can explicitly see the Lagrangian almost toric fibration, which completes the proof of Theorem~\ref{thm:main2}. 

%-----------------------------------------------------------------------
\subsection{Mutation moves in handlebody pictures} \label{subsec:mutate}
%-----------------------------------------------------------------------
Motivated by moves in almost toric geometry, we will show that if two Markov triples $(p_1,p_2,p_3)$ and $(p_1',p_2',p_3')$ are related by a mutation (see Section~\ref{subsec:Markov}), then $X_{p_1,p_2,p_3}$ and $X_{p_1',p_2',p_3'}$ are diffeomorphic. This in turn shows that all the $X_{p_1,p_2,p_3}$ are diffeomorphic to $\C P^2$, thus giving another proof of Theorem~\ref{thm:sympCP2} independent of the work of Lisca and Parma \cite{LiscaParma:horiz2}. 

\begin{proposition}\label{prop:mutate}
  The manifolds $X_{p_1,p_2,p_3}$ and $X_{p_2,p_3, p_1'}$, constructed in Section~\ref{subsec:pants}, corresponding to two mutation-related Markov triples $(p_1, p_2, p_3)$ and $(p_2,p_3, p_1')$, where $p_1' = 3p_2p_3 - p_1$, are diffeomorphic. Also, $X_{p_1,p_2,p_3}$ and $X_{p_1,p_3, p_2'}$, corresponding to two mutation-related Markov triples $(p_1, p_2, p_3)$ and $(p_1,p_3, p_2')$, where $p_2' = 3p_1p_3-p_2$, are diffeomorphic. 
\end{proposition}

We will first show that Proposition~\ref{prop:mutate} implies Theorem~\ref{thm:sympCP2}, independent of Section~\ref{sec:top}.

\begin{proof}[Proof of Theorem~\ref{thm:sympCP2}]
  Consider $X_{1,1,1}$. We first show that it is diffeomorphic to $\cp$. We are using $x=1$ and $y=0$ in the proof of Proposition~\ref{prop:markovarith} and obtain $q_1 = 0$ and $q_2 = q_3 = 3$. Thus the three rational homology balls used in the construction of $X_{1,1,1}$ are $B_{1,0}, B_{1,3},$ and $B_{1,3}$. Each of these is diffeomorphic to $B^4$ since in the handle description of $B_{p,q}$ in Figure~\ref{fig:BpqWeinstein} the $2$-handle will cancel the $1$-handle. Following the construction from Section~\ref{subsec:cap}, we attach a $2$-handle along $(1,-1)$-torus knot in $L(1,2) = \bd B_{1,3} \cong \bd B^4 = S^3$ with the torus framing. We can see that it is an unknot in $S^3$ with framing $+1$ with respect to the Seifert framing. After that, we attach the upside down handlebody $B_{1,1}\,\natural\,B_{1,3} \cong B^4$ and clearly the resulting manifold is $\cp$.
  
  Notice that for the two mutations, $(p_2,p_3,p_1')$ is a left child of $(p_1,p_2,p_3)$ and $(p_1,p_3,p_2')$ is a right child of $(p_1,p_2,p_3)$ in the Markov tree as shown in Figure~\ref{fig:markov-tree}. Thus for any $X_{p_1,p_2,p_3}$, we have a sequence of mutations that relates the Markov triples $(p_1,p_2,p_3)$ and $(1,1,1)$. By repeated application of Proposition~\ref{prop:mutate}, we see that $X_{p_1,p_2,p_3}$ is diffeomorphic to $\cp$. Then, as in the proof of Theorem~\ref{thm:main1}, by the result of Taubes~\cite[Theorem~0.3]{Taubes:SWGr}, it follows that $X_{p_1,p_2,p_3}$ is symplectomorphic to $\cp$ after scaling the symplectic form.
\end{proof}

Proposition~\ref{prop:mutate} will follow from the next lemma. Recall the pants cobordism $W_{p_1,q_1,p_2,q_2}$ defined in Section~\ref{subsec:convert}. We define $Z_{p_1,q_1,p_2,q_2}$ to be the union of $W_{p_1,q_1,p_2,q_2}$, $B_{p_1,q_1}$ and $B_{p_2,q_2}$ as described in Figure~\ref{fig:pantsconvert}.

\begin{lemma}\label{lem:mutate}
  Let $p_1, p_2, p_1', p_2' ,p_3$ be as in the statement of Proposition~\ref{prop:mutate} and $q_1,q_2$ be defined as in Proposition~\ref{prop:markovarith} for $(p_1,p_2,p_3)$. Let $q_1'$ be the appropriate number for the embedding corresponding to the $(p_2,p_3,p_1')$ triple, coming from Proposition~\ref{prop:markovarith}. Define $q_2'$ similarly. Then, the three smooth manifolds $Z_{p_1,p_2,q_1,q_2}$, $Z_{p_2',p_1,q_2',q_1}$ and $Z_{p_1',p_2,q_1',-q_2}$ are diffeomorphic.
\end{lemma}

\begin{proof}
  We first show that $Z_{p_1,p_2,q_1,q_2}$ and $Z_{p_1',p_2,q_1',-q_2}$ are diffeomorphic. Denote a $(p, q)$ torus knot in $\partial(S^1 \times D^3)$ by $K_{p,q}$. Consider the coordinates on $S^1 \times D^3$ to be $(\theta, x, y, z)$. Then $\partial(S^1 \times D^3) = \{(\theta, x, y, z) \mid x^2 + y^2 + z^2 = 1\}$. We can interpret these coordinates in Figure~\ref{fig:pantsconvert}, without loss of generality, as follows: $\theta$ is along the direction of the dotted circle, and $x$ points out of the page. Then Figure~\ref{fig:pantsconvert} represents that $K_{p_2,-q_2}$ lives on a Heegaard torus $$T_2 \coloneqq \{(\theta, x_2, y, z) \mid y^2 + z^2 = 1 - x_2^2\}$$ and $K_{p_1,q_1}$ lives on a Heegaard torus $$T_1 \coloneqq \{(\theta, x_1, y, z) \mid y^2 + z^2 = 1 - x_1^2\}$$ such that $x_2 < x_1$. Consider the self-diffeomorphism $\phi: S^1 \times D^3 \to S^1 \times D^3$ given by $\phi(\theta, x, y, z) = (\theta, -x, y, -z)$. Then, $\phi(K_{p_2,-q_2}) = K_{p_2,q_2}$. Similarly, $\phi(K_{p_1,q_1}) = K_{p_1,-q_1}$. However the order of the $x$-coordinates of the Heegaard tori they live on have been flipped. To get the knots in the same order as the manifold in Figure~\ref{fig:handlebody}, $\phi(K_{p_1,q_1}) = K_{p_1,-q_1}$ needs to be slid past the surgery on $\phi(K_{p_2,-q_2}) = K_{p_2,q_2}$.

  Recall, Lemma~\ref{lem:surfacetwist} says, a $(-1)$-surgery on a knot on the Heegaard torus is the same as cutting and regluing by a $(+1)$-Dehn twist along the knot. So pushing  $\phi(K_{p_1,q_1}) = K_{p_1,-q_1}$ past the Heegaard torus on which $\phi(K_{p_2,-q_2}) = K_{p_2,q_2}$ sits will result in the following torus knot $(p_1',q_1')$ as the following computation, which heavily uses Proposition~\ref{prop:markovarith} , shows:
  \begin{align*}
    \begin{pmatrix} 
      1+p_2q_2&-q_2^2 \\ 
      p_2^2&1-p_2q_2  
    \end{pmatrix} 
    \begin{pmatrix} -q_1\\ p_1  \end{pmatrix} 
    &= \begin{pmatrix} -p_1q_2^2 - q_1 - p_2q_2q_1 \\p_1 - p_2q_2p_1 - p_2^2q_1  \end{pmatrix}\\
    &= - \begin{pmatrix}  3q_2p_3+q_1\\ 3p_2p_3 - p_1 \end{pmatrix}\\
    &=  - \begin{pmatrix} q_1' \\ p_1' \end{pmatrix}
  \end{align*}
  The only computation above that does not directly follow from the definition of mutation and Proposition~\ref{prop:markovarith} is that $3q_2p_3+q_1=q_1'$. Recall that the $p_i$'s and $q_i$'s satisfy the conditions in Proposition~\ref{prop:markovarith}. Note that $q_1'=2p_3 y'$ where $y'=3p_3x+y$ since $p_1'(-x) + p_2y' = p_1'(-x) + p_2(3p_3x+y) = (3p_2p_3 - p_1)(-x) + 3p_2p_3x + p_2y = 1$ and so
  \begin{align*}
    q_1' &= 3p_3y'\\
    &=3p_3(3p_3x + y)\\
    &=3q_2p_3 + q_1
  \end{align*}
  Thus verifying that the conditions in Proposition~\ref{prop:markovarith} are satisfied, it follows that the manifold obtained after the diffeomorphism is exactly $Z_{p_1',p_2,q_1',-q_2}$.
 
  To prove the diffeomorphism between $Z_{p_1,p_2,q_1,q_2}$ and $Z_{p_2',p_1,q_2',q_1}$, we need to slide the knot $K_{p_2,-q_2}$ in $Z_{p_1,p_2,q_1,q_2}$ past the surgery on $K_{p_1,q_1}$. This results in exactly the $(p_2',q_2')$ torus knot, by doing the same matrix arguments as above but interchanging $p_1$ and $p_2$, and $q_1$ and $q_2$. The resultant manifold is then $Z_{p_2',p_1,q_2',q_1}$. 
\end{proof}

\begin{proof}[Proof of Proposition~\ref{prop:mutate}]
Notice that $Z_{p_1,p_2,q_1,q_2}$, $Z_{p_2',p_1,q_2',q_1}$, and $Z_{p_1',p_2,q_1',-q_2}$, all have boundary $L(p_3^2,-p_3q_3+1)$. Thus gluing in the homology ball $B_{p_3,q_3}$ bounded by $L(p_3^2,p_3q_3-1)$ extends the diffeomorphisms and identifies $X_{p_1,p_2,p_3}, X_{p_1,p_3, p'_2}, X_{p_2,p_3, p'_1}$. 
\end{proof}

%-------------------------------------------------------------------------------------
\subsection{Almost toric pictures of \texorpdfstring{$\cp$}{CP2}}\label{subsec:toric}
%-------------------------------------------------------------------------------------
Symington \cite{Symington:four} introduced {\it almost toric manifolds} as a way to "see" certain symplectic manifolds as their images under a moment map; this generalized the pre-existing notion of toric manifolds. The idea is to define a Lagrangian fibration structure on a symplectic manifold $(M, \omega)$ where a regular fiber is a Lagrangian torus, some of the fibers are pinched tori referred to as {\em nodal singularities}. 
More precisely:

\begin{definition}(Vianna \cite[Definition~2.9]{Vianna:exotic})
  An {\em almost toric fibration} of a symplectic four manifold
  $(M, \omega)$ is a Lagrangian fibration $\pi: (M, \omega) \to B$ such that any point of
  $(M, \omega)$ has a Darboux neighborhood (with symplectic form $dx_1 \wedge dy_1 + dx_2 \wedge dy_2$) in which the map $\pi$ has one of the following forms:
  \begin{align*}
    &\pi(x,y) = (x_1,x_2), &\text{regular point}\\
    &\pi(x,y) = (x_1,x_1^2 + x_2^2), &\text{elliptic, corank one}\\
    &\pi(x,y) = (x_1^2+x_2^2,x_2^2+y_2^2), &\text{elliptic, corank two}\\
    &\pi(x,y) = (x_1y_1+x_2y_2,x_1y_2-x_2y_1), &\text{nodal or focus-focus}
  \end{align*}
  with respect to some choice of coordinates near the image point in $B$. An
  {\em almost toric manifold} is a symplectic manifold equipped with an almost toric fibration. A {\em toric fibration} is a Lagrangian fibration induced by an effective Hamiltonian torus action.
\end{definition}

We call the image of each nodal singularity a {\it node}. We will now discuss how to reconstruct the symplectic manifold from the base $B$. Figure~\ref{fig:atf-markov} is an example of $B$ for an almost toric fibration, ignore the blue curves for now. 

\begin{figure}[htbp]{\scriptsize
  \vspace{0.2cm}
  \begin{overpic}[tics=20]{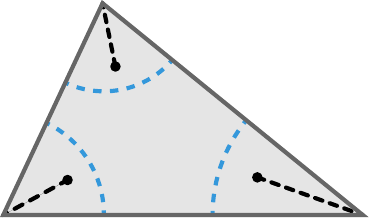}
  \put(85,80){\color{lblue}{$L(p_3^2, p_3q_3 - 1)$}}
  \put(124,48){\color{lblue}{$L(p_1^2, p_1q_1 - 1)$}}
  \put(-32,45){\color{lblue}{$L(p_2^2, p_2q_2 - 1)$}}
  \end{overpic}}
  \vspace{0.3cm}
  \caption{An almost toric picture of $\cp$ corresponding to the Markov triple $(p_1, p_2, p_3)$.}
  \label{fig:atf-markov}
\end{figure}

The pre-image of a regular point in the interior of the triangle is a Lagrangian torus, while the pre-image of a point in the interior of an edge is an isotropic circle (these are elliptic, corank one points). As one approaches a point on the interior of an edge from the interior of the polytope all circles of a fixed slope in the torus collapses to leave the circle above the edge. The circle that collapses is given by the integral normal vector to the line.  The pre-image of a vertex that does not touch a dotted line is a point (these are elliptic, corank two points), the preimage of a vertex that does touch a dotted line is a circle and the pre-image of a node is a pinched torus. The dotted lines from the nodes to the vertices encode the slope of the curve on a regular torus fiber that collapses in the singular fiber corresponding to that node.  This dotted line is called an {\em eigenline}. Its eigendirection  $\pm(a,b)$ encodes the monodromy $A_{(a,b)}~=~\begin{pmatrix} 1-ab &a^2\\ -b^2 &1+ab \end{pmatrix}$ in the affine structure of the base when going around the node -- the eigenline, as the name suggests, is invariant under the monodromy, while the slopes of the boundary on either side of the point where the eigenline hits the boundary, should be related by counterclockwise rotation given by the monodromy. This means that the slope on the right, after being rotated by $A_{(a,b)}$, should match the slope on the left. More specifically, consider the toric diagram and cut the manifold along the pre-image of the dotted line, and then reglue the fibers by the affine transformation $A_{(a,b)}$. This will make a small neighborhood of a corner with a dotted line into an $S^1\times D^3$ since the dotted line cuts the neighborhood into two $S^1\times D^3$ and when crossing the dotted line these two $S^1\times D^3$'s are glued along an $S^1\times D^3$ in their boundaries so that the $T^2$ fibration is preserved. When that neighborhood is expanded to contain the node, one attaches a $2$-handle with framing $-1$ less than the torus framing to the $(a,b)$ sloped curve sitting on a torus in $S^1\times S^2$. Thus we see a neighborhood of a dotted line is a rational homology ball with boundary a lens space. 

{\bf Toric Moves.} We discuss three moves that can be done to an almost toric diagram without changing the manifold. With these three moves we can reproduce Vianna's embeddings of rational homology balls associated to a Markov triple into $\cp$. 

\noindent
{\bf Nodal trade.}  In Figure~\ref{fig:nt-ns} we see the standard toric diagram for a Darboux ball $B^4$ and an almost toric picture that can easily be seen to be $B^4$ as well and can also be shown to be a Darboux ball as well. A {\em nodal trade} is the operation of exchanging one picture for the other. See \cite[Section 6]{Symington:four} for more details. 

\begin{figure}[htbp]{\scriptsize
  \vspace{0.2cm}
  \begin{overpic}[tics=20]{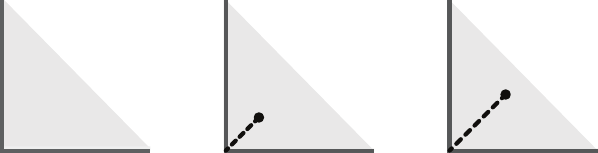}
  \end{overpic}}
  \vspace{0.3cm}
  \caption{On the left is the standard toric picture for a Darboux ball. In the middle is an almost toric picture for the Darboux ball. A nodal trade exchanges one of these pictures for the other. Going between the middle and the right hand picture is a nodal slide. }
  \label{fig:nt-ns}
\end{figure}

\noindent
{\bf Nodal slide.} A {\em nodal slide} simply lengthens or shortens an eigenline. See \cite[Section 6]{Symington:four} for more details. 

\noindent
{\bf Transfer the cut.} The description here follows \cite{Symington:four, Vianna:tori}, and the reader should refer there for more details. Recall a node in the base diagram of an almost toric diagram has an associated eigenline in some eigendirection $(a,b)$, this is the dotted line in the diagram. Notice that there are two line segments leaving the node $x$ in the direction $(a,b)$, the original eigenline $E$ and a line segment $L$ on the opposite side of $x$. 

\begin{figure}[htbp]{\scriptsize
  \vspace{0.2cm}
  \begin{overpic}[tics=20]{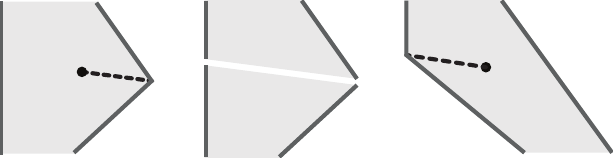}
  \end{overpic}}
  \vspace{0.3cm}
  \caption{Transfer the cut. On the left we see a portion of an almost toric diagram. In the middle we have cut the diagram along the eigendirection for the node. On the right we have reglued the two pieces so that the eigenline now leave the opposite side of the node.}
  \label{fig:tcut}
\end{figure}

One can cut $B$ along $E\cup L$, and apply an affine transformation to one of the pieces so that the vertex that $E$ touched becomes the interior point of an edge and $L$ will now be an eigenline connecting the node $x$ to a corner in the new base diagram. See Figure~\ref{fig:tcut}. 

We can now start with the standard toric picture for $\cp$, see the upper left in Figure~\ref{fig:atf-slide}. We can now perform nodal trades at each corner point to obtain an almost toric diagram for $\cp$.  

Near each corner, one can see $S^3$ as the preimage of the boundary of a neighborhood of the corner, thus this is the almost toric picture corresponding to the Markov triple $(1,1,1)$. One can then perform an operation called  {transferring the cut} (and nodal slide). 

\begin{figure}[htbp]{\scriptsize
  \vspace{0.2cm}
  \begin{overpic}[tics=20]{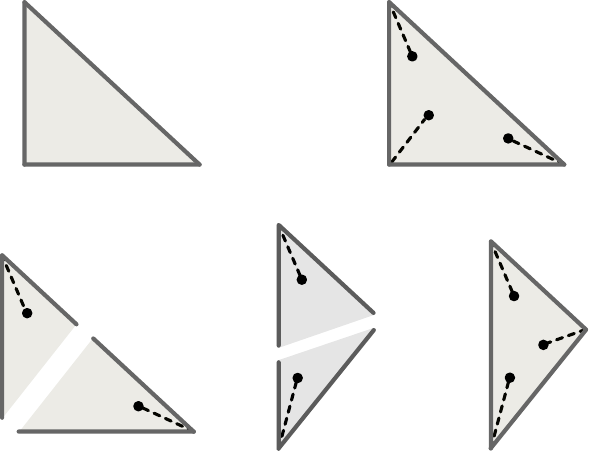}
  \end{overpic}}
  \vspace{0.3cm}
  \caption{In the upper left is the standard toric picture of $\cp$. In the upper right we have performed nodal trades to get an almost toric diagram for $\cp$ for the Markov triple $(1,1,1)$ (a neighborhood of each dotted line is $B^4$). In the bottom left we cut along the diagram along eigenline for the bottom left node. In the middle figure we applied the monodromy to the bottom piece of the diagram. On the bottom right we see the result of the transferring the cut and nodal slide, which gives the mutated Markov triple (1,1,2).}
  \label{fig:atf-slide}
\end{figure}

We show this in Figure~\ref{fig:atf-slide}, this changes the almost toric picture to one where one of the corners now represents $B_{2,1}$, with boundary $L(4,1)$ --- this corresponds to the Markov triple $(1,1,2)$. In general, this procedure allows one to build an almost toric picture of $\cp$ corresponding to any Markov triple $(p_1, p_2, p_3)$ --- this follows from the fact that all Markov triples can be obtained via mutations from $(1,1,1)$. Further, as shown in Figure~\ref{fig:atf-markov}, this picture also encodes the embedding of $\sqcup_{i=1}^3 B_{p_i. q_i}$ into $\cp$. In the next subsection we will see how to draw handlebody decompositions associated to almost toric pictures and give a handlebody interpretation of transferring the cut and see that the proof of Lemma~\ref{lem:mutate} is simply transferring the cut. 

%--------------------------------------------------------------------------------------
\subsection{From almost toric pictures to handlebody decompositions.}\label{sec:atf2hbd}
%--------------------------------------------------------------------------------------
We will first understand a handle decomposition of the complement of a rational homology ball associated to a nodal singularity.
\begin{proposition}\label{complement}
  Consider $\cp$ with the almost toric structure corresponding to the Markov triple $(p_1,p_2,p_3)$ shown in Figure~\ref{fig:atf-markov}. The complement of the rational homology ball $B_{p_3,q_3}$ has a handle decomposition given in Figure~\ref{fig:pantsconvert1} and hence also Figure~\ref{fig:pantsconvert}.
\end{proposition}

\begin{proof}
  As discussed in the previous section, the rational homology ball associated to each nodal singularity has a handle decomposition with one handle in each index $0$, $1$, and $2$. Moreover the $2$-handle is attached to a $(p_i,q_i)$ torus knot in $S^1\times S^2$. So these are precisely the $B_{p_i, q_i}$ from Section~\ref{subsec:rhb}. Thus we understand handle decompositions of the the parts of $\cp$ above the regions separated off of the almost toric diagram by the dotted blue curves in Figure~\ref{fig:atf-round}.
  \begin{figure}[htbp]{\scriptsize
    \vspace{0.2cm}
    \begin{overpic}[tics=20]{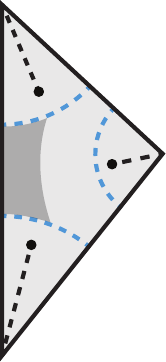}
    \end{overpic}}
    \vspace{0.3cm}
    \caption{The round 1-handle seen in the almost toric picture.}
    \label{fig:atf-round}
  \end{figure}
  If one adds the region $H$ above the dark grey portion in Figure~\ref{fig:atf-round} to the two rational homology balls that it touches, then the result will be diffeomorphic to the complement of the third rational homology ball. 

Notice that $H$ is $[0,1]\times S^1\times D^2$ and it is attached to the two rational homology balls by gluing $\{0\}\times S^1\times D^2$ to one of the homology balls and gluing $\{1\}\times S^1\times D^2$ to the other homology ball. This is called a \dfn{round $1$-handle}. It is easy to see that the circles to which the round $1$-handle is attached consist of rational unknots in the two lens spaces (that is cores of Heegaard tori for the lens spaces) and that the framing on each is the zero framing (the toric structure frames the rational unknots and the attaching regions of the handle). Now, a round $1$-handle can be decomposed into a standard $1$-handle and a standard $2$-handle. The $1$-handle is attached to points on each of the rational unknots and cancels one of the $0$-handles of one of the rational homology balls. This gives Figure~\ref{fig:pantsconvert1} without the yellow curve. Notice that after the $1$-handle is attached, the $2$-handle will be attached to the connect sum of the rational unknots, and this is exactly the yellow curve in Figure~\ref{fig:pantsconvert1}. Moreover, since the round $1$-handle is attached to the neighborhoods of attaching circles using the zero framing on each, we see  
%the zero framings of the attaching regions give us the fact 
that the yellow $2$-handle should have framing $0$. That is Figure~\ref{fig:pantsconvert1} indeed does describe the complement of one of the rational homology balls as claimed. 
\end{proof}

We now have a third proof of Theorem~\ref{thm:sympCP2}, and Theorem~\ref{thm:main2}, that the manifold $X_{p_1,p_2,p_3}$ we constructed in Section~\ref{subsec:pants} are diffeomorphic to $\cp$.

\begin{corollary}\label{3rdp}
  The manifolds $X_{p_1,p_2,p_3}$ constructed in Section~\ref{subsec:pants} are diffeomorphic to $\cp$.
\end{corollary}

\begin{proof}
  Given a Markov triple $(p_1,p_2,p_3)$ the rational homology balls in $X_{p_1,p_2,p_3}$ and in $\cp$ with the almost toric structure associated to the triple are the same. Moreover, the previous proposition shows us that the complement of $B_{p_3,q_3}$ in each are both obtained by attaching the same round $1$-handle to $B_{p_1,q_1}\cup B_{p_2,q_2}$ and thus they are diffeomorphic. Since any diffeomorphism of $\partial B_{p_3,q_3}$ extends over $B_{p_3,q_3}$, see Item~(2) in Remark~\ref{rem:pushingpast}, we know that $X_{p_1,p_2,p_3}$ is diffeomorphic to $\cp$.  
\end{proof}

{\bf Transferring the cut in handlebody diagrams of $\cp$.} Notice that when transferring the cut, only two of the nodal singularities are involved, and hence only two of the rational homology balls. Specifically, there is the node that is being transferred and there is the node that has an affine transformation applied to it (the third node is unaffected). So we only need to consider the complement of one of the rational homology balls when studying transfer the cut. We know from Proposition~\ref{complement} that this complement is $S^1\times D^3$ with two $2$-handles attached, where the attaching circles of the $2$-handles correspond to the eigenlines of the nodes. Given an almost toric picture for $\cp$ the attaching curves of the $2$-handles occur on separate Heegaard tori in $S^1\times S^2$ and the order of those tori is important and is determined by the order on the nodal singularities. Transferring the cut corresponds to changing this order, but when one does this, one must apply the associated monodromy $A_{(p_i,q_i)}$ to the other attaching circle, this corresponds to the handle slide in the proof of Lemma~\ref{lem:mutate}. Similarly, one can read the curves in the opposite order (this corresponds to the diffeomorphism $\phi$ in the proof of Lemma~\ref{lem:mutate}). This is the almost toric geometry inspiration for the proof of Lemma~\ref{lem:mutate}.

%----------------------------------------------------------------------
\subsection{From handle descriptions to almost toric pictures}\label{sec:hbd2atf}
%----------------------------------------------------------------------
Recall from Section~\ref{subsec:pants}  that $X_{p_1,p_2,p_3}$ admits a symplectic handlebody decomposition for each Markov triple $(p_1, p_2, p_3)$. The following proposition shows that $X_{p_1,p_2.p_3}$ admits a smooth fibration by tori (and three singular tori) and after a deformation of the symplectic structure, we can arrange that this is a Lagrangian (almost toric) fibration. 

\begin{proposition}\label{prop:toric}
  For each Markov triple $(p_1,p_2,p_3)$, after a deformation of the symplectic structure $\omega_{p_1,p_2,p_3}$ there is a smooth map from $X_{p_1,p_2,p_3}$ to $\R^2$, with the generic pre-image of a point being a Lagrangian torus, and the identification of $X_{p_1,p_2,p_3}$ with $\cp$ in Corollary~\ref{3rdp}. This fibration agrees with the almost toric fibration of $\cp$ corresponding to the Markov triple $(p_1,p_2,p_3)$.
\end{proposition}

\begin{proof}
  We know that each $B_{p_i,q_i}$ admits an almost toric fibration by Lagrangian tori from our discussion in Section~\ref{subsec:toric}. Moreover we know the round $1$-handle discussed in the proof of Proposition~\ref{complement} admits a toric fibration. The round $1$-handle used in our construction of $X_{p_1,p_2,p_3}$ possibly has a different symplectic structure on it, but it will nonetheless have a smooth toric fibration that extends the toric fibrations on the $B_{p_i,q_i}$. Thus the diffeomorphism from $X_{p_1,p_2,p_3}$ to $\cp$ in Corollary~\ref{3rdp} takes torus fibers to torus fibers. The pull-back of the symplectic from on $\cp$ will be deformation equivalent to the one we constructed $\omega_{p_1,p_2,p_3}$ by Taubes theorem \cite[Theorem~0.3]{Taubes:SWGr} and this completes the proof. 
\end{proof}

This establishes that the two descriptions of $\cp$, coming from $X_{p_1,p_2,p_3}$ and the almost toric picture corresponding to the Markov triple, are analogous. 

\begin{proof}[Proof of Theorem~\ref{thm:main2}]
  The statement of Theorem~\ref{thm:main2} follows from Proposition~\ref{prop:toric}.
\end{proof}

\subsection{General strategy} The results of our paper suggest a general strategy for constructing a symplectic handlebody description of any closed symplectic manifold that admits an almost toric fibration with its base being a convex polygon. For instance, Del Pezzo surfaces endowed with a monotone symplectic form admit such almost toric fibrations \cite{Vianna:tori}. 

For brevity, we describe the strategy informally, following the notation in \cite[Section 5]{Symington:four}, while providing references and citations for the reader's convenience. Let $X$ be an almost toric manifold with base $\mathcal{B} \subset \mathbb{R}^2$, where $\mathcal{B}$ is a convex polygon. Consider a hexagon $S \subset \mathcal{B}$ where the three edges are part of the edges of $\mathcal{B}$ and the other three are in the interior of $\mathcal{B}$. Assume further that $S$ does not contain any node and eigenline; see the region inside the blue dotted curves in Figure~\ref{fig:atf-markov} for example. By the discussion in \cite[Section 9]{Symington:four}, each edge of the hexagon $S$ that is contained in the interior of $\mathcal{B}$ represents a universally tight lens space. Denote them by $L(p_i,q_i)$, $1 \leq i \leq 3$. It follows from the convex condition on $\mathcal{B}$ that the components of $\mathcal{B} \setminus S$ are symplectic fillings of these lens spaces, call them $W_i$ for $1 \leq i \leq 3$. 

The discussion in Section~\ref{sec:geometry} can be adapted to build a (concave) pants cobordism among these three lens spaces, as shown in Figure~\ref{fig:pants}. Denote this cobordism by $C$. Now we can construct a symplectic handlebody diffeomorphic to $X$ by attaching the pants cobordism $C$ to $W_i$. For the handle decomposition of $W_i$ that comes from an almost toric fibration, refer to \cite[Section 9.3]{evans}; see also \cite[Section 2.8]{etnyre-roy} which reinterprets Lisca's classification of fillings of universally tight lens spaces \cite{Lisca:fillings}.

Lastly, since there is a unique symplectic structure on del Pezzo surfaces up to deformation equivalence \cite{LL:delPezzo1,LL:delPezzo2} (see also \cite{Salamon:Uniqueness}), the remaining arguments should be identical to the ones in Section~\ref{sec:hbd2atf}. 

\bibliography{references}
\bibliographystyle{plain}
\end{document}